\documentclass[twocolumn]{svjour3}          
\usepackage[utf8]{inputenc}
\usepackage{amsmath, amssymb}
\smartqed  

\usepackage{xcolor}
\usepackage{graphicx}
\usepackage[outercaption]{sidecap} 
\usepackage[caption=false]{subfig} 

\usepackage{enumitem} 
\usepackage{url}
\usepackage{algpseudocode}
\usepackage{algorithm}

\newtheorem{assumption}{Assumption}

\title{Inexact Derivative-Free Optimization for Bilevel Learning\thanks{MJE acknowledges support from the EPSRC (EP/S026045/1, EP/T026693/1), the Faraday Institution (EP/T007745/1) and the Leverhulme Trust (ECF-2019-478).}
}

\author{Matthias J. Ehrhardt         \and
        Lindon Roberts
}

\institute{M. J. Ehrhardt \at
              Institute for Mathematical Innovation and Department of Mathematical Sciences \\
              University of Bath \\
              \email{m.ehrhardt@bath.ac.uk}
           \and
           L. Roberts \at
              Mathematical Sciences Institute \\
              Australian National University \\
              \email{lindon.roberts@anu.edu.au}
}

\date{Received: date / Accepted: date}

\newcommand{\scn}[2]{#1\mathrm{e}{#2}}
\newcommand{\R}{\mathbb{R}}
\renewcommand{\t}[1]{\widetilde{#1}} 

\newcommand{\kappaef}{\kappa_{\rm ef}}
\newcommand{\kappaeg}{\kappa_{\rm eg}}
\newcommand{\bigO}{\mathcal{O}}

\newcommand{\revision}[1]{{#1}}

\usepackage[style=numeric-comp, sortcites=true, giveninits=true, sorting=none,
url=false, isbn=false, doi=false]{biblatex}
\AtEveryBibitem{\clearfield{month}}
\AtEveryCitekey{\clearfield{month}}
\renewbibmacro{in:}{}

\def\LowObj{\Phi}
\def\LowObjI{\LowObj_i}
\def\LowObjIT{\LowObj_{i,\theta}}

\def\SpaceX{X}
\def\SpaceT{\Theta}

\def\HeightGr{3cm}

\newcommand{\secref}[1]{Section~\ref{#1}}

\newcommand{\thmref}[1]{Theorem~\ref{#1}}
\newcommand{\propref}[1]{Proposition~\ref{#1}}
\newcommand{\lemref}[1]{Lemma~\ref{#1}}
\newcommand{\corref}[1]{Corollary~\ref{#1}}
\renewcommand{\algref}[1]{Algorithm~\ref{#1}}
\newcommand{\assref}[1]{Assumption~\ref{#1}}
\newcommand{\figref}[1]{Figure~\ref{#1}}

\begin{document}

\maketitle

\begin{abstract}
Variational regularization techniques are dominant in the field of mathematical imaging. A drawback of these techniques is that they are dependent on a number of parameters which have to be set by the user. A by now common strategy to resolve this issue is to learn these parameters from data. While mathematically appealing this strategy leads to a nested optimization problem (known as bilevel optimization) which is computationally very difficult to handle. \revision{It is common when solving the upper-level problem to assume access to exact solutions of the lower-level problem, which is practically infeasible.} In this work we propose to solve these problems using inexact derivative-free optimization algorithms which never require \revision{exact lower-level problem solutions, but instead assume access to approximate solutions with controllable accuracy, which is achievable in practice}. We \revision{prove} global convergence and \revision{a worst-case complexity bound for } our approach. \revision{We test} our proposed framework on ROF-denoising and learning MRI sampling patterns.
Dynamically adjusting the lower-level accuracy yields  learned parameters with similar reconstruction quality as high-accuracy evaluations but with dramatic reductions in computational work (up to 100 times faster in some cases).
\end{abstract}

\section{Introduction}
Variational regularization techniques are dominant in the field of mathematical imaging. For example, when solving a linear inverse problem $Ax = y$, variational regularization can be posed as the solution to
\begin{align}
    \min_{x} \mathcal D(Ax, y) + \alpha \mathcal R(x) \, . \label{EQ:LOWER:MOTIVATION}
\end{align}
Here the data fidelity $\mathcal D$ is usually chosen related to the assumed noise model of the data $y$ and the regularizer $\mathcal R$ models our a-priori knowledge of the unknown solution. Many options have been proposed in the literature, see for instance \cite{Ito2014book,Benning2018actanumerica,Chambolle2016actanumerica,Arridge2019,Scherzer2008book} and references therein. An important parameter for any variational regularization technique is the regularization parameter $\alpha$. While some theoretical results and heuristic choices have been proposed in there literature, see e.g. \cite{Engl1996,Benning2018actanumerica} and references therein or the L-curve criterion \cite{Hansen1992lcurve}, the appropriate choice of the regularization parameter in a practical setting remains an open problem. Similarly, other parameters in \eqref{EQ:LOWER:MOTIVATION} have to be chosen by the user, such as smoothing of the total variation  \cite{Chambolle2016actanumerica}, the hyperparameter for total generalized variation \cite{Bredies2010} or the sampling pattern in magnetic resonance imaging (MRI), see e.g. \cite{Usman2009,Gozcu2018samplingmri,Sherry2019sampling}.

Instead of using heuristics for choosing all of these parameters, here we are interested in finding these from data. A by-now common strategy to learn parameters of a variational regularization model from data is bilevel learning, see e.g. \cite{DeLosReyes2013,Kunisch2013bilevel,Ochs2015,Hintermueller2020tgv,Sherry2019sampling,Riis2018,Bartels2020} and references in \cite{Arridge2019}
.  Given labelled data $(x_i, y_i)_{i=1,\ldots,n}$ we find parameters $\theta \in \SpaceT \subset \mathbb R^m$ by solving the upper-level problem
\begin{align}
	\min_{\theta\in\SpaceT} f(\theta) := \frac1n \sum_{i=1}^n \|\hat x_i(\theta) - x_i\|^2 + \mathcal J(\theta), \label{EQ:UPPER}
\end{align}
where $\hat x_i(\theta) \in \SpaceX \subset \mathbb R^d$ \revision{aims to recover the true data $x_i$ by solving} the lower-level problems
\begin{align}
    \hat x_i(\theta) := \arg \min_{x\in\SpaceX} \LowObjIT(x), \qquad \forall i = 1, \ldots, n \label{EQ:LOWER} \, .
\end{align}
The lower-level objective $\LowObjIT$ could be of the form $\LowObjIT(x) = \mathcal D(Ax, y_i) + \theta \mathcal R(x)$ as in \eqref{EQ:LOWER:MOTIVATION} but we will not restrict ourselves to this special case. In general $\LowObjIT$ will depend on the data $y_i$. 

\revision{In many situations, it is possible to acquire suitable data $(x_i, y_i)_{i=1,\ldots,n}$. For image denoising, we may take any ground truth images $x_i$ and add artificial noise to generate $y_i$. Alternatively, if we aim to learn a sampling pattern (such as for learning MRI sampling patterns, which we consider in this work), then $x_i$ can be any fully sampled image. The same also holds for problems such as image compression, where again $x_i$ is any ground truth image. In both these cases, $y_i$ is subsampled information from $x_i$ (depending on $\theta$) from which the remaining information is reconstructed to get $\hat{x}_i(\theta)$.}

While mathematically appealing, this nested optimization problem is computationally very difficult to handle since even the evaluation of the upper-level problem \eqref{EQ:UPPER} requires the exact solution of the lower-level problems \eqref{EQ:LOWER}. This requirement is practically infeasible and common algorithms in the literature compute the lower-level solution only to some accuracy, thereby losing any theoretical performance guarantees, see e.g. \cite{DeLosReyes2013,Riis2018,Sherry2019sampling}.
One reason for needing exact solutions is to compute the gradient of the upper-level objective using the implicit function theorem \cite{Sherry2019sampling}, which we address by using upper-level solvers which do not require gradient computations.

In this work we propose to solve these problems using inexact derivative-free optimization~(DFO) algorithms which never require exact solutions to the lower-level problem while still yielding convergence guarantees. 
Moreover, by dynamically adjusting the accuracy we gain a significant computational speed-up compared to using a fixed accuracy for all lower-level solves. 
The proposed framework is tested on two problems: learning regularization parameters for ROF-denoising and learning the sampling pattern in MRI.

\revision{We contrast our approach to \cite{Kunisch2013bilevel}, which develops a semismooth Newton method to solve the full bilevel optimality conditions. In \cite{Kunisch2013bilevel} the upper- and lower-level problems are of specific structure, and exact solutions of the (possibly very large) Newton system are required. Separately, the approach in \cite{Ochs2015} replaces the lower-level problem with finitely many iterations of some algorithm and solves this perturbed problem exactly. Our formulation is very general and all approximations are controlled to guarantee convergence to the solution of the original variational problem.}

\textbf{Aim:} Use inexact computations of $\hat x_i(\theta)$ within a derivative-free upper-level solver, which makes \eqref{EQ:UPPER} computationally tractable, while retaining convergence guarantees.

\subsection{Derivative-free optimization}
Derivative-free optimization methods---that is, optimization methods that do not require access to the derivatives of the objective (and/or constraints)---have grown in popularity in recent years, and are particularly suited to settings where the objective is computationally expensive to evaluate and/or noisy; we refer the reader to \cite{Conn2009,Audet2017} for background on DFO and examples of applications, and to \cite{Larson2019} for a comprehensive survey of recent work.
The use of DFO for algorithm tuning has previously been considered in a general framework \cite{Audet2006a}, and in the specific case of hyperparameter tuning for neural networks in \cite{Lakhmiri2019}. 

Here, we are interested in the particular setting of learning for variational methods \eqref{EQ:UPPER}, which has also been considered in \cite{Riis2018} where a new DFO algorithm based on discrete gradients has been proposed. In \cite{Riis2018} it was assumed that the lower-level problem can be solved exactly such that the bilevel problem can be reduced to a single nonconvex optimization problem. In the present work we lift this stringent assumption. 

In this paper we focus on DFO methods which are adapted to nonlinear least-squares problems as analyzed in \cite{Zhang2010,Cartis2019a}.
These methods are so-called `model-based', in that they construct a model approximating the objective at each iteration, locally minimize the model to select the next iterate, and update the model with new objective information.
Our work also connects to \cite{Conn2012}, which considers model-based bilevel optimization where both the lower- and upper-level problems are solved in a derivative-free manner; particular attention is given here to reusing evaluations of the (assumed expensive) lower-level objective at nearby upper-level parameters, to make lower-level model construction simpler.

Our approach for bilevel DFO is is based on dynamic-accuracy (derivative-based) trust-region methods \cite[Chapter 10.6]{Conn2000}.
In these approaches, we use the measures of convergence (e.g.~trust-region radius, model gradient) to determine a suitable level of accuracy with which to evaluate the objective; we start with low accuracy requirements, and increase the required accuracy as we converge to a solution.
In a DFO context, this framework is the basis of \cite{Conn2012}, and a similar approach was considered in \cite{Chen2012} in the context of analyzing protein structures.
This framework has also been recently extended in a derivative-based context to higher-order regularization methods \cite{Bellavia2019,Gratton2019}.
We also note that there has been some work on multilevel and multi-fidelity models (in both a DFO and derivative-based context), where an expensive objective can be approximated by surrogates which are cheaper to evaluate \cite{March2012,Calandra2019}.

\subsection{Contributions}
There are a number of novel aspects to this work.
Our use of DFO for bilevel learning means our upper-level solver genuinely expects inexact lower-level solutions.
We give worst-case complexity theory for our algorithm both in terms of upper-level iterations and computational work from the lower-level problems.
Our numerical results on ROF-denoising and a new framework for learning MRI sampling patterns demonstrate our approach is substantially faster---up to 100 times faster---than the same DFO approach with high accuracy lower-level solutions, while achieving the same quality solutions.
More details on the different aspects of our contributions are given below.

\paragraph{Dynamic accuracy DFO algorithm for bilevel learning}
As noted in \cite{Sherry2019sampling}, bilevel learning can require very high-accuracy solutions to the lower-level problem.
We avoid this via the introduction of a dynamic accuracy model-based DFO algorithm.
In this setting, the upper-level solver dynamically changes the required accuracy for lower-level problem minimizers, where less accuracy is required in earlier phases of the upper-level optimization.
The proposed algorithm is similar to \cite{Conn2012}, but adapted to the nonlinear least-squares case and allowing derivative-based methods to solve the lower-level problem.
\revision{Our theoretical results extend the convergence results of \cite{Conn2012} to include derivative-based lower-level solvers and a least-squares structure, as well as adding a worst-case complexity analysis in a style similar to \cite{Cartis2019a} (which is also not present in the derivative-based convergence theory in \cite{Conn2000}). This analysis gives bounds on the number of iterations of the upper-level solver required to reach a given optimality, which we then extend to bound the total computational effort required for the lower-level problem solves. There is increasing interest, but comparatively fewer works, which explicitly bound the total computational effort of nonconvex optimization methods; see \cite{Royer2020} for Newton-CG methods and references therein.}
We provide a preliminary argument that our computational effort bounds are tight with regards to the desired upper-level solution accuracy, although we delegate a complete proof to future work.

\paragraph{Robustness}
We observe in all our results using several lower-level solvers (gradient descent and FISTA) for a variety of applications that the proposed upper-level DFO algorithm converges to similar objective values and minimizers. We also present numerical results for denoising showing that the learned parameters are robust to initialization of the upper-level solver despite the upper-level problem being likely nonconvex. Together, these results suggest that this framework is a robust approach for bilevel learning.

\paragraph{Efficiency}
Bilevel learning with a DFO algorithm was previously considered \cite{Riis2018}, but there a different DFO method based on discrete gradients was used, and was applied to nonsmooth problems with exact lower-level evaluations.
In \cite{Riis2018}, only up to two parameters were learned, whereas here we demonstrate our approach is capable of learning many more. Our numerical results include examples with up to 64 parameters.

We demonstrate that the dynamic accuracy DFO achieves comparable or better objective values than the fixed accuracy variants and final reconstructions of comparable quality.
However our approach is able to achieve this with a dramatically reduced computational load, in some cases up to 100 times less work than the fixed accuracy variants.

\paragraph{New framework for learning MRI sampling}
We introduce a new framework to learn the sampling pattern in MRI based on bilevel learning.
Our idea is inspired by the image inpainting model of \cite{Chen2014}.
Compared to other algorithms to learn the sampling pattern in MRI based on first-order methods \cite{Sherry2019sampling}, the proposed approach seems to be much more robust to \revision{initialization} and choice of solver for the lower-level problem.
As with the denoising examples, our dynamic accuracy DFO achieves the same upper-level objective values and final reconstructions as fixed accuracy variants but with substantial reductions in computational work.

\paragraph{Regularization parameter choice rule with machine learning}
Our numerical results suggest that the bilevel framework can learn regularization parameter choice rule which yields a convergent regularization method in the sense of \cite{Scherzer2008book,Ito2014book}, indicating for the first time that machine learning can be used to learn mathematically sound regularization methods.

\subsection{Structure}
In \secref{sec_lower_level} we describe problems where the lower-level model \eqref{EQ:LOWER:MOTIVATION} applies and describe how to efficiently attain a given accuracy level using standard first-order methods.
Then in \secref{sec_dfo} we introduce the dynamic accuracy DFO algorithm and present our global convergence and worst-case complexity bounds.
Finally, our numerical experiments are described in \secref{sec_numerics}.

\subsection{Notation}
Throughout, we let $\|\cdot\|$ we denote the Euclidean norm of a vector in $\R^n$ and the operator 2-norm of a matrix in $\R^{m\times n}$. 
We also define the weighted (semi)norm $\|x\|^2_S := x^T Sx$ for a symmetric and positive (semi)definite matrix $S$.
The gradient of a scalar-valued function $f:\R^n\to \R$ is denoted by $\nabla f:\R^n\to\R^n$, and the derivative of a vector-valued function $f : \mathbb R^n \to \mathbb R^m$ is denoted by $\partial f : \mathbb R^n \to \mathbb R^{n \times m}, (\partial f)_{i,j} = \partial_i f_j$ where $\partial_i f_j$ denotes the partial derivative of $f_j$ with respect to the $i$th coordinate. 
If $f$ is a function of two variables $x$ and $y$, then $\partial_x f$ denotes the derivative with respect to $x$. 

\subsection{Software}
Our implementation of the DFO algorithm and all numerical testing code will be made public upon acceptance.

\section{Lower-Level Problem} \label{sec_lower_level}

In order to have sufficient control over the accuracy of the solution to \eqref{EQ:LOWER} we will assume that $\LowObjIT$ are $L_i$-smooth and $\mu_i$-strongly convex, see definitions below.
\begin{definition}[Smoothness]
A function $f : \mathbb R^n \to \mathbb R$ is $L$-smooth if it is differentiable and its derivative is Lipschitz continuous with constant \revision{$L>0$}, i.e. for all $x, y \in \mathbb R^n$ we have
$\|\nabla f(x) - \nabla f(y)\| \leq L \|x - y\|$.
\end{definition}
\begin{definition}[Strong Convexity]
A function $f : \mathbb R^n \to \mathbb R$ is $\mu$-strongly convex \revision{for $\mu>0$} if $f - \frac{\mu}{2}\|\cdot\|^2$ is convex.
\end{definition}

Moreover, when the lower-level problem is strictly convex and smooth, with $\LowObjI(x, \theta) := \LowObjIT(x)$ we can equivalently describe the minimizer \revision{of $\LowObjIT$} by 
\begin{align}
\partial_x \LowObjI(\hat x_i(\theta), \theta) = 0 \, .
\end{align}
Smoothness properties of $\hat x_i$ follow from the implicit function theorem and its generalizations if $\LowObjI$ is smooth and regular enough.

\begin{assumption} \label{as:phi}
We assume that for all $i = 1, \ldots, n$ the following statements hold.
\begin{enumerate}[leftmargin=*]
    \item Convexity: For all $\theta \in \SpaceT$ the functions $\LowObjIT$ are $\mu_i$-strongly convex.
    \item Smoothness in $x$: For all $\theta \in \SpaceT$ the functions $\LowObjIT$ are $L_i$-smooth.
    \item Smoothness in $(x, \theta)$: The derivatives $\partial_x \LowObjI : \SpaceX \times \SpaceT \to \SpaceX$ and $\partial_{xx} \LowObjI : \SpaceX \times \SpaceT \to \SpaceX^2$ exist and are continuous.
\end{enumerate}
\end{assumption}

\begin{theorem} \label{th:phi}
Under \assref{as:phi} the function $\hat x_i(\theta) := \arg \min_x \LowObjI(x, \theta)$ is 
\begin{enumerate}[leftmargin=*]
\item well-defined 
\item locally Lipschitz
\item continuously differentiable and \\ $\partial \hat x_i(\theta) = - \partial_{xx}\LowObjI(\hat x_i(\theta), \theta)^{-1} \partial_{x}\partial_\theta \LowObjI(\hat x_i(\theta), \theta)$.
\end{enumerate}
\begin{proof}
Ad 1) Finite and convex functions are continuous \cite[Corollary 2.36]{Rockafellar2008}. It is easy to show that $\mu$-strongly convex functions are coercive. Then the existence and uniqueness follows from classical theorems, e.g. \cite[Theorem 6.31]{Bredies2018book}. Ad 2) This statement follows directly from \cite[Theorem 2.1]{Robinson1980}. Ad 3) This follows directly from the classical inverse function theorem, see e.g.~\cite[Theorem 3.5.1]{Duistermaat2004a}.
\end{proof}
\end{theorem}

\subsection{Examples}
A relevant case of the model introduced above is the parameter tuning for linear inverse problems, which can be solved via the variational regularization model
\begin{align}
\frac12 \|A x - y_i\|^2_S + \alpha \operatorname{TV}(x), \label{EQ:LOWER:LINIP}
\end{align}
where $\operatorname{TV}(x) := \sum_{j=1}^m \|\widehat{\nabla} x(j)\|$ denotes the discretized total variation, e.g.~$\widehat{\nabla} x(j)$
is the finite forward difference discretization of the spatial gradient of $x$ at pixel $j$.
\revision{However, we note that \eqref{EQ:LOWER:LINIP} does not satisfy \assref{as:phi}.}

\revision{To ensure \assref{as:phi} holds, we instead use} $\|x\| \approx \sqrt{\|x\|^2 + \nu^2}$, \revision{to} approximate problem \eqref{EQ:LOWER:LINIP} by a smooth and strongly convex problem of the form
\begin{align}
\hat x_i(\theta) &:= \arg \min_x \left\{\LowObjIT(x) = \frac12 \|A(\theta)x - y_i\|^2_{S(\theta)}\right. \nonumber \\ & \qquad\qquad \left.+ \alpha(\theta) \operatorname{TV}_{\nu(\theta)}(x) + \frac{\xi(\theta)}{2} \|x\|^2 \right\} \label{EQ:LOWER:APPROX} \, ,
\end{align}
with the smoothed total variation given by $\operatorname{TV}_{\nu(\theta)}(x) := \sum_{j=1}^m \sqrt{\|\widehat{\nabla} x(j)\|^2 + \nu(\theta)^2}$. 
Here we already introduced the notation that various parts of the problem may depend on a vector of parameters $\theta$ which usually needs to be selected manually. 
We will learn these parameters using the bilevel framework. 
For simplicity denote $A_\theta := A(\theta)$, $S_\theta := S(\theta)$, $\alpha_\theta := \alpha(\theta)$, $\nu_\theta := \nu(\theta)$ and $\xi_\theta := \xi(\theta)$. Note that $\LowObjIT$ in \eqref{EQ:LOWER:APPROX} is $L_i$-smooth and $\mu_i$-strongly convex with
\begin{align}
L_i &\leq \|A_\theta^* S_\theta A_\theta\| + \alpha_\theta \frac{\|\partial\|^2}{\nu_\theta} + \xi_\theta, \quad \text{and} \nonumber \\ 
\mu_i &\geq \lambda_{\operatorname{min}}(A_\theta^* S_\theta A_\theta) + \xi_\theta \, , \label{eq_condition_number}
\end{align}
where $\lambda_{\operatorname{min}}(A_\theta^* S_\theta A_\theta)$ denotes the smallest eigenvalue of $A_\theta^* S_\theta A_\theta$ \revision{and $A_\theta^*$ is the adjoint of $A_\theta$}.

\revision{We now describe two specific problems we will use in our numerical results. They both choose a specific form for \eqref{EQ:LOWER} which aims to find a minimizer $\hat{x}_i(\theta)$ which (approximately) recovers the data $x_i$, and so both use \eqref{EQ:UPPER} as the upper-level problem.}

\subsubsection{Total Variation-based Denoising} \label{sec_tv_denoising}
A particular problem we consider is a smoothed version of the ROF model \cite{Rudin1992ROF}, i.e. $A_\theta = I, S_\theta = I$. Then \eqref{EQ:LOWER:APPROX} simplifies to
\begin{align}
\LowObjIT(x) = \frac12 \|x - y_i\|^2 + \alpha_\theta \operatorname{TV}_{\nu_\theta}(x) + \frac{\xi_\theta}{2} \|x\|^2, \label{EQ:LOWER:APPROX:TV}
\end{align}
which is $L_i$-smooth and $\mu_i$-strongly convex with constants as in \eqref{eq_condition_number} with $\|A_\theta^* S_\theta A_\theta\| = \|I\| = 1$ and $\lambda_{\operatorname{min}}(A_\theta^* S_\theta A_\theta) = \lambda_{\operatorname{min}}(I) = 1$.
In our numerical examples we will consider two cases. First, we will just learn the regularization parameter $\alpha$ given manually set $\nu$ and $\xi$. Second, we will learn all three parameters $\alpha, \nu$ and $\xi$.

\subsubsection{Undersampled MRI Reconstruction}
Another problem we consider is the reconstruction from undersampled MRI data, see e.g. \cite{Lustig2007sparseMRI}, which can be phrased as \eqref{EQ:LOWER:APPROX} with $A_\theta = F$ where $F$ is the discrete Fourier transform and $S_\theta = \operatorname{diag}(s), s \in [0, 1]^d$. Then \eqref{EQ:LOWER:APPROX} simplifies to
\begin{align}
\LowObjIT(x) = \frac12 \|F x - y_i\|^2_{S_\theta} + \alpha_\theta \operatorname{TV}_{\nu_\theta}(x) + \frac{\xi_\theta}{2} \|x\|^2, \label{EQ:LOWER:APPROX:MRI}
\end{align}
which is $L_i$-smooth and $\mu_i$-strongly convex with constants as in \eqref{eq_condition_number} with $\|A_\theta^* S_\theta A_\theta\| \leq 1$ and $\lambda_{\operatorname{min}}(A_\theta^* S_\theta A_\theta) \geq 0$. The sampling coefficients $s_j$ indicate the relevance of a sampling location. The data term \eqref{EQ:LOWER:APPROX:MRI} can be rewritten as
\begin{align}
\|F x - y_i\|^2_{S_\theta} = \sum_{s_j > 0} s_j |[F x - y_i]_j|^2 \, . 
\end{align}
Most commonly the values $s$ are binary and manually chosen. Here we aim to use bilevel learning to find a sparse $s$ such that the images $x_i$ can be reconstructed well from sparse samples of $y_i$. This approach was first proposed in \cite{Sherry2019sampling}.

\subsection{Example Training Data} \label{sec_data_model}
Throughout this paper, we will consider training data of artificially-generated 1D images.
Each ground truth image $x_i$ is randomly-generated piecewise-constant function.
For a desired image size $N$, we select values $C_i\in[N/4,3N/4]$ and $R_i\in[N/8,N/4]$ from a uniform distribution.
We then define $x_i\in\R^N$ by
\begin{align}
    [x_i]_j := \begin{cases}1, & |j-C_i| < R_i, \\ 0, & \text{otherwise}, \end{cases} \qquad \forall j=1,\ldots,N.
\end{align}
That is, each $x_i$ is zero except for a single randomly-generated subinterval of length $2 R_i$ centered around $C_i$ where it takes the value 1.

We then construct our $y_i$ by taking the signal to be reconstructed and adding Gaussian noise.
Specifically, for the image denoising problem we take 
\begin{align}
    y_i := x_i + \sigma \omega_i \,,
\end{align}
where $\sigma>0$ and $\omega_i\in\R^N$ is randomly-drawn vector of i.i.d.~standard Gaussians.
For the MRI sampling problem, we take
\begin{align}
    y_i := F x_i + \frac{\sigma}{\sqrt{2}}\omega_i \,.
\end{align}
where $\sigma>0$ and $\omega_i\in\mathbb{C}^N$ is a randomly-drawn vector with real and imaginary parts both standard Gaussians.

In \figref{fig:1ddenoising:data} we plot an example collection of pairs $(x_i,y_i)$ for the image denoising problem with $N=256$, and in \figref{fig:1ddenoising:results} we plot the solution to \eqref{EQ:LOWER:APPROX:TV} for the first of these $(x_i,y_i)$ pairs for a variety of choices for the parameters $\alpha_{\theta},\epsilon_{\theta},\eta_{\theta}$.

\begin{figure}
    \centering
    \includegraphics[width=8cm, height=3.5cm]{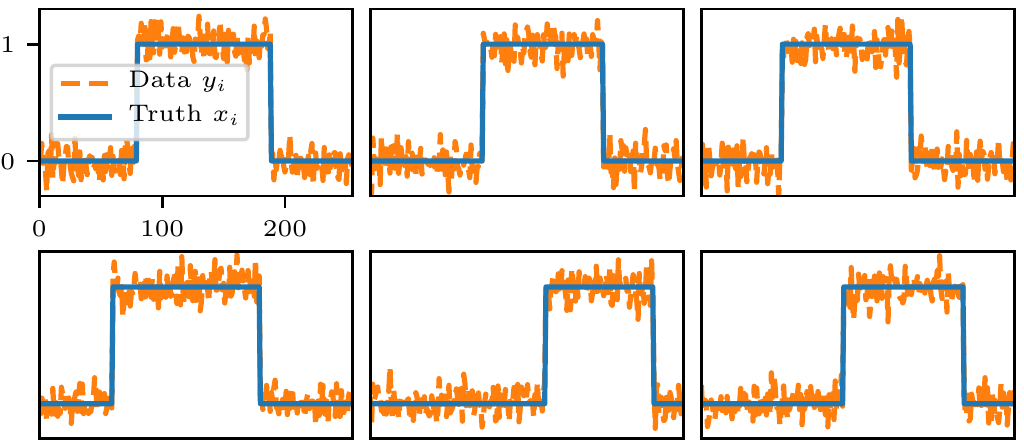}
    \caption{Examples of training pairs $(x_i, y_i)$ for image denoising.}
    \label{fig:1ddenoising:data}
\end{figure}

\subsection{Approximate Solutions}

\subsubsection{Gradient Descent}

For simplicity we drop the dependence on $i$ for the remainder of this section.

The lower-level problem \eqref{EQ:LOWER} can be solved with gradient descent (GD) which converges linearly for $L$-smooth and $\mu$-strongly convex problems. One can show (e.g. \cite{Chambolle2016actanumerica}) that  GD
\begin{align}
x^{k+1} = x^k - \tau \nabla \Phi(x^k), \label{EQ:GRADDESC}
\end{align}
with $\tau = 1/L$, converges linearly to the unique solution $x^*$ of \eqref{EQ:LOWER}. More precisely, for all $k \in \mathbb N$ we have \cite[Theorem 10.29]{Beck2017}
\begin{align}
\|x^k - x^*\|^2 \leq \left(1 - \mu / L\right)^k \|x^0 - x^*\|^2 \, .\label{EQ:TOL}
\end{align}
Moreover, if one has a good estimate of the strong convexity constant $\mu$, then it is better to choose $\tau=2/(L+\mu)$, which gives an improved linear rate \cite[Theorem 2.1.15]{Nesterov2004}
\begin{align}
    \|x^k - x^*\|^2 \leq (1 - \mu / L)^{2k} \|x^0 - x^*\|^2. \label{EQ:TOLGDUNCONS}
\end{align}

\begin{figure}
    \centering
    \includegraphics[width=8cm, height=4.3cm]{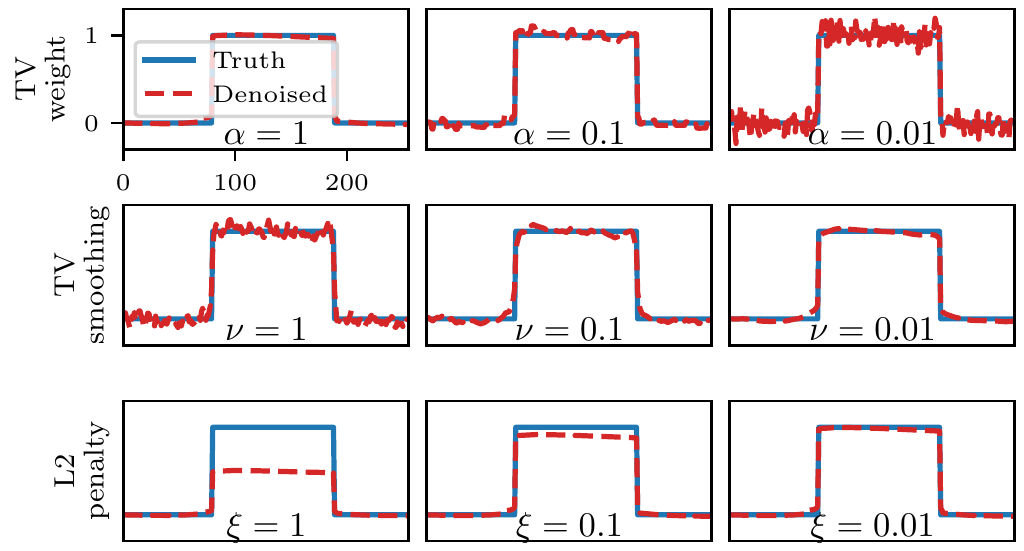}
    \caption{Examples of denoised data using model \eqref{EQ:LOWER:APPROX:TV} obtained by running GD with a tolerance of $\|x^k-x^*\|\leq\scn{1}{-6}$. The data $(x_i,y_i)$ is the top-left image in \figref{fig:1ddenoising:data}. \textbf{Top:} results with $\alpha=1,0.1,0.01$ (left to right) with $\nu=\xi=\scn{1}{-3}$ throughout.
    \textbf{Middle:} results with $\nu=1,0.1,0.01$ (left to right) with $\alpha=1$ and $\xi=\scn{1}{-3}$ throughout.
    \textbf{Bottom:} results with $\xi=1,0.1,0.01$ (left to right) with $\alpha=1$ and $\nu=\scn{1}{-3}$ throughout.}
    \label{fig:1ddenoising:results}
\end{figure}

\subsubsection{FISTA}
Similarly, we can use FISTA \cite{Beck2009} to approximately solve the lower-level problem. FISTA applied to a smooth objective with convex constraints is a modification of \cite{Nesterov1983} and can be formulated as the iteration
\begin{align}\begin{aligned}
t_{k+1} &= \frac{1 - q t_k^2 + \sqrt{(1 - q t_k^2)^2 + 4 t_k^2}}{2}, \\
\beta_{k+1} &= \frac{(t_k - 1)(1 - t_{k+1} q)}{t_{k+1}(1 - q)}, \\
z^{k+1} &= x^k + \beta_{k+1}(x^k - x^{k-1}), \\
x^{k+1} &= z^{k+1} - \tau \nabla \Phi(z^{k+1}), \\
\end{aligned}\label{EQ:FISTA}
\end{align}
where $q:=\tau \mu$, and we choose $\tau=1/L$ and $t_0=0$ \cite[Algorithm 5]{Chambolle2016actanumerica}.
We then achieve linear convergence with \cite[Theorem 4.10]{Chambolle2016actanumerica}
\begin{align}
    \Phi(x^k)-\Phi(x^*) \leq \left(1-\sqrt{q}\right)^k \left[\frac{L}{2}(1+\sqrt{q})\|x^0-x^*\|^2\right],
\end{align}
and so, since $\Phi(x^k)-\Phi(x^*)\geq (\mu/2)\|x^k-x^*\|^2$ from $\mu$-strong convexity, we get
\begin{align}
    \|x^k - x^\ast\|^2 \leq \left(1 - \sqrt{\frac{\mu}{L}}\right)^k \left[\frac{L}{\mu}\left(1+\sqrt{\frac{\mu}{L}}\right) \|x^0 - x^*\|^2\right] \, . \label{eq_fista_rate}
\end{align}

\subsubsection{Ensuring accuracy requirements}
We will need to be able to solve the lower-level problem to sufficient accuracy that we can guarantee $\|x^k-x^*\|^2 \leq \epsilon$, for a suitable accuracy $\epsilon>0$.
We can guarantee this accuracy by ensuring we terminate with $k$ sufficiently large, given an estimate $\|x^0-x^*\|^2$, using the a-priori bounds \eqref{EQ:TOL} or \eqref{eq_fista_rate}.
A simple alternative is to use the a-posteriori bound  $\|x-x^*\|\leq\|\nabla\Phi(x)\|/\mu$ for all $x$ \revision{(a consequence of \cite[Theorem 5.24(iii)]{Beck2009})}, and terminate once
\begin{align}
    \|\nabla\Phi(x^k)\|^2/\mu^2 \leq\epsilon. \label{eq_accuracy_gradient}
\end{align}
To compare these two options, we consider two test problems: i) a version of Nesterov's quadratic \cite[Section 2.1.4]{Nesterov2004} in $\R^{10}$, and ii) 1D image denoising. Nesterov's quadratic is defined as
\begin{align}
    \Phi(x) &:= \frac{\t{\mu}(Q-1)}{8}\left(x^T A x - 2x_1\right) + \frac{\t{\mu}}{2}\|x\|^2, \nonumber \\ & \text{where}\: A:=\begin{bmatrix}2 & -1 & &  \\ -1 & 2 & \ddots & \\ & \ddots & \ddots & -1 \\ & & -1 & 2\end{bmatrix},
\end{align}
for $x\in\R^{10}$, with $\t{\mu}=1$ and $Q=100$, which is $\mu$-strongly convex and $L$-smooth for $\mu\approx 3$ and $L\approx 98$; we apply no constraints, $\SpaceX=\R^{10}$.

\begin{figure}
\def\WidthGr{4.1cm}
    \centering
    \subfloat[GD]{\label{fig_accuracy_gd}\includegraphics[width=\WidthGr,height=2.85cm]{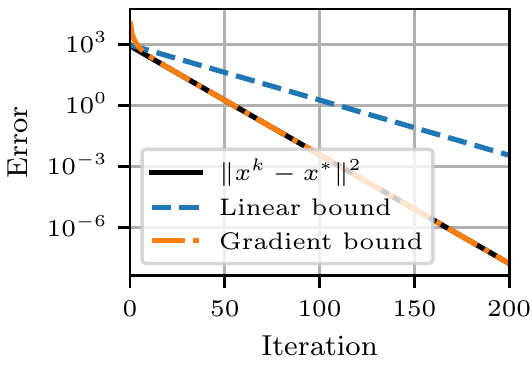}}
    \hfill
    \subfloat[FISTA]{\label{fig_accuracy_fista}\includegraphics[width=\WidthGr,height=2.85cm]{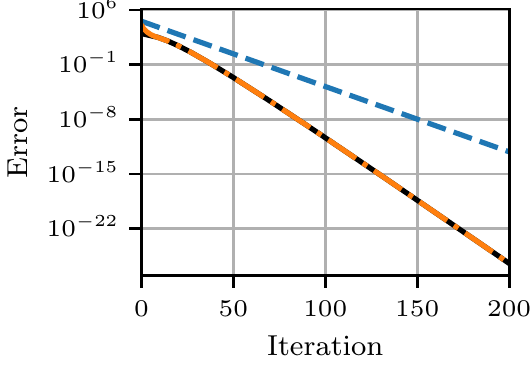}}
    \hfill
    \subfloat[GD]{\label{fig_accuracy_gd_desnoising}\includegraphics[width=\WidthGr,height=2.85cm]{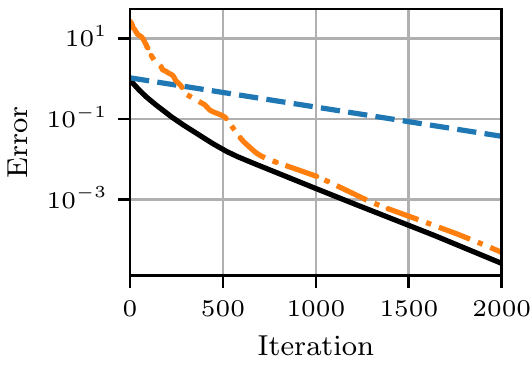}}
    \hfill
    \subfloat[FISTA]{\label{fig_accuracy_fista_desnoising}\includegraphics[width=\WidthGr,height=2.85cm]{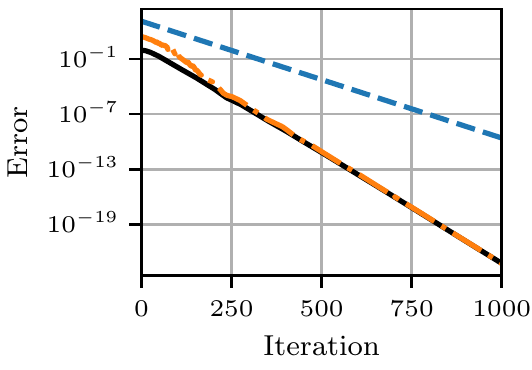}}
	\caption{Comparison of a-priori linear convergence bounds \eqref{EQ:TOLGDUNCONS} and \eqref{eq_fista_rate} against the a-posteriori gradient bound \eqref{eq_accuracy_gradient}. 
	\textbf{(a, b):} 200 iterations of GD and FISTA on Nesterov's quadratic function. \textbf{(c, d):} GD (2,000 iterations) and FISTA (1,000 iterations) on a 1D denoising problem.}
	\label{fig_accuracy_bounds}
\end{figure}

We also consider a 1D denoising problem as in \eqref{EQ:LOWER:APPROX:TV} with randomly-generated data $y\in\R^N$ (with $N=100$ pixels) as per Section \ref{sec_data_model}, $\alpha=0.3$, $\nu=\xi=10^{-3}$, and $x^*$ estimated by running $10^4$ iterations of FISTA.
Here, the problem is $\mu$-strongly convex and $L$-smooth with $\mu\approx 1$ and $L\approx 1,201$. 
We estimate the true solution $x^*$ by running FISTA for 10,000 iterations (which gives an upper bound estimate $\|x^k-x^*\|^2\leq \scn{3}{-26}$ from \eqref{eq_accuracy_gradient}).

In \figref{fig_accuracy_bounds}, we compare the true error $\|x^k-x^*\|^2$ against the a-priori linear convergence bounds \eqref{EQ:TOL} or \eqref{eq_fista_rate} with the true value of $\|x^0-x^*\|^2$, and the a-posteriori gradient bound \eqref{eq_accuracy_gradient}.
In both cases, the gradient-based bound \eqref{eq_accuracy_gradient} provides a much tighter estimate of the error, particularly for high accuracy requirements. 
Thus, in our numerical results, we terminate the lower-level solver as soon \eqref{eq_accuracy_gradient} is achieved for our desired tolerance.
The gradient-based bound has the additional advantage of not requiring an a priori estimate of $\|x^0-x^*\|$.
For comparison, in our results below we will also consider terminating GD/FISTA after a fixed number of iterations.

\section{Dynamic Accuracy DFO Algorithm} \label{sec_dfo}

\subsection{DFO Background}
Since evaluating $\hat x_i(\theta)$ in the upper-level problem \eqref{EQ:UPPER} is only possible with some error (it is computed by running an iterative process), it is not straightforward or cheap to evaluate $\partial \hat x_i(\theta)$.
Hence for solving \eqref{EQ:UPPER} we turn to DFO techniques, and specifically consider those which exploit the nonlinear least-squares problem structure.
In this section we outline a model-based DFO method for nonlinear least-squares problems \cite{Cartis2019a}, a trust-region method based on the classical (derivative-based) Gauss--Newton method \cite[Chapter 10]{Nocedal2006}.
However, these approaches are based on having access to exact function evaluations, and so we augment this with a standard approach for dynamic accuracy trust-region methods \cite[Chapter 10.6]{Conn2000}; this was previously considered for general model-based DFO methods in \cite{Conn2012}.

Here, we write the upper-level problem \eqref{EQ:UPPER} in the general form
\begin{align}
	\min_{\theta\in\R^d} f(\theta) := \frac{1}{n}\|r(\theta)\|^2 = \frac{1}{n}\sum_{i=1}^n r_i(\theta)^2, \label{eq_upper_exact}
\end{align}
where $r_i(\theta):= \|\hat{x}_i(\theta)-x_i\|$ and $r(\theta):=[r_1(\theta), \ldots, r_n(\theta)]^T$.
Without loss of generality, we do not include a regularization term $\mathcal{J}(\theta)$; we can incorporate this term by defining $r_{n+1}(\theta):=\sqrt{\mathcal{J}(\theta)}$ and then taking $r(\theta):=[r_1(\theta), \ldots, r_{n+1}(\theta)]^T$, for instance.

The upper-level objective \eqref{eq_upper_exact} assumes access to exact evaluations of the lower-level objective $r_i(\theta)$, which is not achievable in practice.
We therefore assume we only have access to inaccurate evaluations $\t{x}_i(\theta)\approx \hat{x}_i(\theta)$, giving $\t{r}_i(\theta):= \|\t{x}_i(\theta)-x_i\|$, $\t{r}(\theta):=[\t{r}_1(\theta), \ldots, \t{r}_n(\theta)]^T$, and $\t{f}(\theta):=\frac{1}{n}\|\t{r}(\theta)\|^2$.

Our overall algorithmic framework is based on trust-region methods, where at each iteration $k$ we construct a model $m^k$ for the objective which we hope is accurate in a neighborhood of our current iterate $\theta^k$.
Simultaneously we maintain a trust-region radius $\Delta^k>0$, which tracks the size of the neighborhood of $\theta^k$ where we expect $m^k$ to be accurate.
Our next iterate is determined by minimizing the model $m^k$ within a ball of size $\Delta^k$ around $\theta^k$.

Usually $m^k$ is taken to be a quadratic function (e.g.~a second-order Taylor series for $f$ about $\theta^k$).
However here we use the least-squares problem structure \eqref{eq_upper_exact} and construct a linear model
\begin{align}
	r(\theta^k+s) \approx \t r(\theta^k+s) \approx M^k(s) := \t{r}(\theta^k) + J^k s, \label{EQ:VECTOR_MODEL}
\end{align}
where $\t{r}(\theta^k)$ is our approximate evaluation of $r(\theta^k)$  and $J^k\in\R^{n\times d}$ is a matrix approximating $\partial r(\theta^k)^T$.
We construct $J^k$ by interpolation: we maintain an interpolation set $z^0,\ldots,z^d\in\R^d$ (where $z^0:=\theta^k$ at each iteration $k$) and choose $J^k$ so that
\begin{align}
	M^k(z^t-\theta^k)=\t{r}(z^t), \qquad \forall t=1,\ldots,d.
\end{align}
\revision{This condition ensures that our linear model $M^k$ exactly interpolates $\t{r}$ at our interpolation points $z^t$ (i.e.~the second approximation in \eqref{EQ:VECTOR_MODEL} is exact for each $s=z^t-\theta^k$).}
We can therefore find $J^k$ by solving the $d\times d$ linear system (with $n$ right-hand sides):
\begin{align}
	\begin{bmatrix}(z^1-\theta^k) & \cdots & (z^d-\theta^k)\end{bmatrix}^T g^k_i = \begin{bmatrix}\t{r}_i(z^1)-\t{r}_i(\theta^k) \\ \vdots \\ \t{r}_i(z^d)-\t{r}_i(\theta^k) \end{bmatrix}, \label{EQ:INTERP_SYSTEM}
\end{align}
for all $i=1,\ldots,n$, where $g^k_i\in\R^d$ is the $i$-th row of $J^k$.
The model $M^k$ gives a natural quadratic model for the full objective $f$:
\begin{align}
	f(\theta^k+s) \approx m^k(s) := \frac{1}{n}\|M^k(s)\|^2 &= \t{f}(\theta^k) + (g^k)^T s \nonumber \\
	& \quad+ \frac{1}{2}s^T H^k s, \label{EQ:SCALAR_MODEL}
\end{align}
where $g^k:= \frac{2}{n} (J^k)^T \t{r}(\theta^k)$ and $H^k:=\frac{2}{n}(J^k)^T J^k$.
We compute a tentative step $s^k$ as a(n approximate) minimizer of the trust-region subproblem
\begin{align}
	\min_{s\in\R^d} m^k(s), \qquad \text{subject to} \qquad \|s\| \leq \Delta^k. \label{EQ:TRS}
\end{align}
There are a variety of efficient algorithms for computing $s^k$ \cite[Chapter 7]{Conn2000}.
Finally, we evaluate $\t{f}(\theta^k+s^k)$ and decide whether to accept or reject the step (i.e.~set $\theta^{k+1}=\theta^k+s^k$ or $\theta^{k+1}=\theta^k$) depending on the ratio
\begin{align}
	\rho^k = \frac{\text{actual reduction}}{\text{predicted reduction}} := \frac{f(\theta^k)-f(\theta^k+s^k)}{m^k(0)-m^k(s^k)}. \label{EQ:RHO_TRUE}
\end{align}
Although we would like to accept/reject using $\rho^k$, in reality we only observe the approximation
\begin{align}
	\t{\rho}^k := \frac{\t{f}(\theta^k)-\t{f}(\theta^k+s^k)}{m^k(0)-m^k(s^k)}, \label{EQ:RHO_NOISY}
\end{align}
and so we use this instead.

This gives us the key components of a standard trust-region algorithm.
We have two extra considerations in our context: the accuracy of our derivative-free model \eqref{EQ:SCALAR_MODEL} and the lack of exact evaluations of the objective.

Firstly, we require a procedure to verify if our model \eqref{EQ:SCALAR_MODEL} is sufficiently accurate inside the trust-region, and if not, modify the model to ensure its accuracy.
We discuss this in Section \ref{sec_geometry}.
The notion of `sufficiently accurate' we use here is that $m^k$ is as good an approximation to $f$ as a first-order Taylor series (up to constant factors), which we call `fully linear'.\footnote{If $f$ is $L$-smooth then the Taylor series $m^k (s)=f(\theta^k)+\nabla f(\theta^k)^T s$ is fully linear with $\kappaef=L/2$ and $\kappaeg=L$ for all $\Delta^k$.}

\begin{definition}[Fully linear model]
	The model $m^k$ \eqref{EQ:SCALAR_MODEL} is a fully linear model for $f(\theta)$ in $B(\theta^k,\Delta^k)$ if there exist constants $\kappaef,\kappaeg>0$ (independent of $\theta^k$ and $\Delta^k$) such that
	\begin{align}
		|f(\theta^k+s)-m^k(s)| &\leq \kappaef(\Delta^k)^2, \\
		\|\nabla f(\theta^k+s) - \nabla m^k(s)\| &\leq \kappaeg\Delta^k,
	\end{align}
	for all $\|s\|\leq\Delta^k$.
\end{definition}

Secondly, we handle the inaccuracy in objective evaluations by ensuring $\t{f}(\theta^k)$ and $\t{f}(\theta^k+s^k)$ are evaluated to a sufficiently high accuracy when we compute $\t{\rho}^k$ \eqref{EQ:RHO_NOISY}.
Specifically, suppose we know that $|\t{f}(\theta^k)-f(\theta^k)|\leq\delta^k$ and $|\t{f}(\theta^k+s^k)-f(\theta^k+s^k)|\leq\delta^k_+$ for some accuracies $\delta^k$ and $\delta^k_+$.
\revision{Throughout, we use $\delta^k$ and $\delta^k_+$ to refer to the accuracies with which $\t{f}(\theta^k)$ and $\t{f}(\theta^k+s^k)$ have been evaluated, in the sense above.}
Before we compute $\t{\rho}^k$, we first ensure that
\begin{align}
	\max(\delta^k, \delta^k_+) \leq \eta_1^\prime \left[m^k(0)-m^k(s^k)\right], \label{EQ:MIN_ACCURACY}
\end{align}
where $\eta_1^{\prime}>0$ is an algorithm parameter.
We achieve this by running the lower-level solver for a sufficiently large number of iterations.

The full upper-level algorithm is given in \algref{alg_dynamic_accuracy_dfo}; it is similar to the approach in \cite{Conn2012}, the DFO method \cite[Algorithm 10.1]{Conn2009}---adapted for the least-squares problem structure---and the (derivative-based) dynamic accuracy trust-region method \cite[Algorithm 10.6.1]{Conn2000}.

\begin{algorithm}[t]
	{\small 
	\begin{algorithmic}[1]
		\Statex \textbf{Inputs:} Starting point $\theta^0\in\R^n$, initial trust region radius $0<\Delta^0\leq \Delta_{\max}$. 
		\vspace{0.2em}
		\Statex \textbf{Parameters:} strictly positive values $\Delta_{\max}, \gamma_{\rm dec},\gamma_{\rm inc}, \eta_1, \eta_2, \eta_1^{\prime}, \epsilon$ satisfying $\gamma_{\rm dec}<1<\gamma_{\rm inc}$, $\eta_1 \leq \eta_2 < 1$, and $\eta_1^{\prime}<\min(\eta_1,1-\eta_2)/2$.
		\vspace{0.5em}
		\State Select an arbitrary interpolation set and construct $m^0$ \eqref{EQ:SCALAR_MODEL}.
		\For{$k=0,1,2,\ldots$}
		    \Repeat
		        \State Evaluate $\t{f}(\theta^k)$ to sufficient accuracy that \eqref{EQ:MIN_ACCURACY} holds with $\eta_1^{\prime}$ (using $s^k$ from the previous iteration of this inner repeat/until loop). Do nothing in the first iteration of this repeat/until loop.
			    \If{$\|g^k\| \leq \epsilon$} \label{ln_crit_start}
			    \State By replacing $\Delta^k$ with $\gamma_{\rm dec}^{i}\Delta^k$ for $i=0,1,2,\ldots$, find $m^k$ and $\Delta^k$ such that $m^k$ is fully linear in $B(\theta^k,\Delta^k)$ and $\Delta^k \leq \|g^k\|$. \hfill \textit{[criticality phase]}
			    \EndIf
			    \State Calculate $s^k$ by (approximately) solving \eqref{EQ:TRS}.
			\Until{the accuracy in the evaluation of $\t{f}(\theta^k)$ satisfies \eqref{EQ:MIN_ACCURACY} with $\eta_1^{\prime}$} \hfill \textit{[accuracy phase]}
			\State Evaluate $\t{r}(\theta^k+s^k)$ so that \eqref{EQ:MIN_ACCURACY} is satisfied with $\eta_1^{\prime}$ for $\t{f}(\theta^k+s^k)$, and calculate $\t{\rho}^k$ \eqref{EQ:RHO_NOISY}. \label{ln_accuracy_end}
			\State Set $\theta^{k+1}$ and $\Delta^{k+1}$ as:
			\begin{align}
				\theta^{k+1} = \begin{cases} \theta^k+s^k, & \text{$\t{\rho}^k\geq\eta_2$, or $\t{\rho}^k\geq\eta_1$ and $m^k$} \\ & \text{fully linear in $B(\theta^k,\Delta^k)$,} \\ \theta^k, & \text{otherwise}, \end{cases}
			\end{align}
			and
			\begin{align}
				\Delta^{k+1} = \begin{cases} \min(\gamma_{\rm inc}\Delta^k,\Delta_{\max}), & \text{$\t{\rho}^k\geq\eta_2$, } \\ \Delta^k, & \text{$\t{\rho}^k<\eta_2$ and $m^k$ not } \\
				& \text{fully linear in $B(\theta^k,\Delta^k)$,} \\ \gamma_{\rm dec}\Delta^k, & \text{otherwise}. \end{cases} \label{eq_delta_update}
			\end{align}
			\State If $\theta^{k+1}=\theta^k+s^k$, then build $m^{k+1}$ by adding $\theta^{k+1}$ to the interpolation set (removing an existing point).
			Otherwise, set $m^{k+1}=m^k$ if $m^k$ is fully linear in $B(\theta^k,\Delta^k)$, or form $m^{k+1}$ by making $m^k$ fully linear in $B(\theta^{k+1},\Delta^{k+1})$.
		\EndFor
	\end{algorithmic}
	} 
	\caption{Dynamic accuracy DFO algorithm for \eqref{eq_upper_exact}.}
	\label{alg_dynamic_accuracy_dfo}
\end{algorithm}

Our main convergence result is the below.

\begin{theorem} \label{thm_iter_complexity}
	Suppose Assumptions \ref{ASS:SMOOTHNESS} and \ref{ass_cauchy_bdd_hess} hold.
	Then if 
	\begin{align}
		K &> \left\lfloor\left(2 + 2\frac{\log\gamma_{\rm inc}}{|\log \gamma_{\rm dec}|}\right)\frac{2 (\kappaeg+1) f(\theta^0)}{(\eta_1-2\eta_1^{\prime})\epsilon\Delta_{\min}}\right. \nonumber \\
		& \qquad\qquad\qquad\qquad\qquad \left. + 2\frac{\log(\Delta^0/\Delta_{\min})}{|\log \gamma_{\rm dec}|}\right\rfloor, \label{eq_kepsilon}
	\end{align}
	with $\kappaeg$ and $\Delta_{\min}$ given by Lemmas \ref{lem_fully_linear} and \ref{lem_delta_min} respectively, then $\min_{k=0,\ldots,K}\|\nabla f(\theta^k)\| < \epsilon$.
\end{theorem}

We summarize \thmref{thm_iter_complexity} as follows, noting that the iteration and evaluation counts match the standard results for model-based DFO (e.g.~\cite{Cartis2019a,Garmanjani2016}).

\begin{corollary} \label{cor_complexity}
	Suppose the assumptions of \thmref{thm_iter_complexity} hold.
	Then \algref{alg_dynamic_accuracy_dfo} is globally convergent; i.e.
	\begin{align}
	    \lim_{k\to\infty}\|\nabla f(\theta^k)\|=0. \label{eq_liminf}
	\end{align}
	Also, if $\epsilon\in(0,1]$, then the number of iterations before $\|\nabla f(\theta^k)\|<\epsilon$ for the first time is at most $\bigO(\kappa^3\epsilon^{-2})$ and the number of evaluations of $\t{r}(\theta)$ is at most $\bigO(d\kappa^3\epsilon^{-2})$, where $\kappa:=\max(\kappaef, \kappaeg, \kappa_H)$.\footnote{If we have to evaluate $\t{r}(\theta)$ at different accuracy levels as part of the accuracy phase, we count this as one evaluation, since we continue solving the corresponding lower-level problem from the solution from the previous, lower accuracy evaluation.}
\end{corollary}

\revision{We note that since $\theta^k \in\mathcal{B}$ and $\mathcal{B}$ is bounded from \assref{ASS:SMOOTHNESS} (and closed by continuity of $f$), then by \corref{cor_complexity} and compactness there exists a subsequence of iterates $\{\theta_{k_i}\}_{i\in\mathbb{N}}$ which converges to a stationary point of $f$. However, there are relatively few results which prove convergence of the full sequence of iterates for nonconvex trust-region methods (see \cite[Theorem 10.13]{Conn2000} for a restricted result in the derivative-based context).}

\subsection{Guaranteeing Model Accuracy} \label{sec_geometry}
As described above, we need a process to ensure that $m^k$ \eqref{EQ:SCALAR_MODEL} is a fully linear model for $f$ inside the trust region $B(\theta^k,\Delta^k)$.
For this, we need to consider the geometry of the interpolation set.

\begin{definition}
	The Lagrange polynomials of the interpolation set $\{z^0,z^1,\ldots,z^d\}$ are the linear polynomials $\ell_t$, $t=0,\ldots,d$ such that $\ell_t(z^s)=\delta_{s,t}$ for all $s,t=0,\ldots,d$.
\end{definition}

The Lagrange polynomials of $\{z^0,\ldots,z^d\}$ exist and are unique whenever the matrix in \eqref{EQ:INTERP_SYSTEM} is invertible.
The required notion of `good geometry' is given by the below definition (where small $\Lambda$ indicates better geometry).

\begin{definition}[$\Lambda$-poisedness]
	For $\Lambda>0$, the interpolation set $\{z^0,\ldots,z^d\}$ is $\Lambda$-poised in $B(\theta^k,\Delta^k)$ if $|\ell_t(\theta^k+s)|\leq\Lambda$ for all $t=0,\ldots,d$ and all $\|s\|\leq\Delta^k$.
\end{definition}

The below result confirms that, provided our interpolation set has sufficiently good geometry, and our evaluations $\t{r}(\theta^k)$ and $\t{r}(y^t)$ are sufficiently accurate, our interpolation models are fully linear.

\begin{assumption} \label{ASS:SMOOTHNESS}
	The extended level set 
	\begin{align}
		\mathcal{B} := \{z: \text{$z\in B(\theta,\Delta_{\max})$ for some $\theta$ with $f(\theta)\leq f(\theta^0)$}\},
	\end{align}
	is bounded, and $r(\theta)$ is continuously differentiable and $\partial r(\theta)$ is Lipschitz continuous with constant $L_J$ in $\mathcal{B}$.
\end{assumption}

In particular, \assref{ASS:SMOOTHNESS} implies that $r(\theta)$ and $\partial r(\theta)$ are uniformly bounded in the same region---that is, $\|r(\theta)\|\leq r_{\max}$ and $\|\partial r(\theta)\| \leq J_{\max}$ for all $\theta\in\mathcal{B}$---and $f$ \eqref{eq_upper_exact} is $L$-smooth in $\mathcal{B}$ \cite[Lemma 3.2]{Cartis2019a}.

\revision{We note that $\mathcal{B}$ in \assref{ASS:SMOOTHNESS} is bounded whenever the regularizer $\mathcal{J}$ is coercive, such as in \secref{sec_learning_mri_results}. This may also be replaced by the weaker assumption that $r$ and $\partial r$ are uniformly bounded on $\mathcal{B}$ (and $\mathcal{B}$ need not be bounded) \cite[Assumption 3.1]{Cartis2019a}, and there are theoretical results which give this for some inverse problems in image restoration \cite{DeLosReyes2016}. In our numerical experiments, we enforce upper and lower bounds on $\theta$, which also yields the uniform boundedness of $r$ and $\partial r$. Also, we note that if $r_i(\theta)=\|\hat{x}_i(\theta)-x_i\|$ is not itself $L$-smooth, we can instead treat each entry of $\hat{x}_i(\theta)-x_i$ as a separate term in \eqref{eq_upper_exact}.}

\begin{lemma} \label{lem_fully_linear}
	Suppose \assref{ASS:SMOOTHNESS} holds and $\Delta^k\leq\Delta_{\max}$. 
	If the interpolation set $\{z^0:=\theta^k,z^1,\ldots,z^d\}$ is $\Lambda$-poised in $B(\theta^k,\Delta^k)$ and for each evaluation $t=0,\ldots,d$ and each $i=1,\ldots,n$ we have
	\begin{align}
		\|\t{x}_i(z^t)-\hat{x}_i(z^t)\|\leq c(\Delta^k)^2, \label{eq_desired_accuracy}
	\end{align}
	for some $c>0$, then the corresponding models $M^k$ \eqref{EQ:VECTOR_MODEL} and $m^k$ \eqref{EQ:SCALAR_MODEL} are fully linear models for $r(\theta)$ and $f(\theta)$ respectively.
\end{lemma}
\begin{proof}
	This is a straightforward extension of \cite[Lemma 3.3]{Cartis2019a}, noting that
	\begin{align}
		|\t{r}_i(z^t)-r_i(z^t)| \leq \|\t{x}_i(z^t)-\hat{x}_i(z^t)\|, \qquad \forall i=1,\ldots,n, \label{EQ:TMP1}
	\end{align}
	and so \eqref{eq_desired_accuracy} gives $\|\t{r}(z^t) - r(z^t)\| \leq c \sqrt{n} (\Delta^k)^2$ for all $t=0,\ldots,d$.
\end{proof}

We conclude by noting that for any $\Lambda>1$ there are algorithms available to determine if a set is $\Lambda$-poised, and if it is not, change some interpolation points to make it so; details may be found in \cite[Chapter 6]{Conn2009}, for instance.

\subsection{Lower-Level Objective Evaluations}
We now consider the accuracy requirements that \algref{alg_dynamic_accuracy_dfo} imposes on our lower-level objective evaluations.
In particular, we require the ability to satisfy \eqref{EQ:MIN_ACCURACY}, which imposes requirements on the error in the calculated $\t{f}$, rather than the lower-level evaluations $\t{r}$.
The connection between errors in $\t{r}$ and $\t{f}$ is given by the below result.

\begin{lemma} \label{lem_inexact_eval_accuracy}
	Suppose we compute $\t{x}_i(\theta)$ satisfying $\|\t{x}_i(\theta)-\hat{x}_i(\theta)\|\leq\delta_x$ for all $i=1,\ldots,n$.
	Then we have
	\begin{align}
		|\t{f}(\theta)-f(\theta)| &\leq 2\sqrt{\t{f}(\theta)}\: \delta_x + \delta_x^2   \nonumber \\ &\text{and} \quad |\t{f}(\theta)-f(\theta)| \leq 2\sqrt{f(\theta)}\: \delta_x + \delta_x^2. \label{EQ:DELTA_BOUND1}
	\end{align}
	Moreover, if $\|\t{x}_i(\theta)-\hat{x}_i(\theta)\| \leq \sqrt{\t{f}(\theta)+\delta_f} - \sqrt{\t{f}(\theta)}$ for $i=1,\ldots,n$, then $|\t{f}(\theta)-f(\theta)|\leq\delta_f$.
\end{lemma}
\begin{proof}
	Letting $\epsilon(\theta):=\t{r}(\theta)-r(\theta)$, we have
    \begin{align}
		f(\theta) &= \frac{1}{n}\|\t{r}(\theta)-\epsilon(\theta)\|^2, \\
		&= \t{f}(\theta) - \frac{2}{n}\epsilon(\theta)^T \t{r}(\theta) + \frac{1}{n}\|\epsilon(\theta)\|^2,
    \end{align}
    and hence
    \begin{align}
		|f(\theta)-\t{f}(\theta)| &\leq \frac{2}{n}\|\epsilon(\theta)\| \|\t{r}(\theta)\| + \frac{1}{n}\|\epsilon(\theta)\|^2, \\
		&\leq \frac{2}{n}\sqrt{n}\|\epsilon(\theta)\|_{\infty}\sqrt{n \t{f}(\theta)} + \|\epsilon(\theta)\|_{\infty}^2,
    \end{align}
	and the first part of \eqref{EQ:DELTA_BOUND1} follows since $\|\epsilon(\theta)\|_{\infty}\leq\delta_x$ from \eqref{EQ:TMP1}.
	The second part of \eqref{EQ:DELTA_BOUND1} follows from an identical argument but writing $\t{f}(\theta)=\frac{1}{n}\|r(\theta)+\epsilon(\theta)\|^2$, and the final conclusion follows immediately from the first part of \eqref{EQ:DELTA_BOUND1}.
\end{proof}

We construct these bounds to rely mostly on $\t{f}(\theta)$, since this is the value which is observed by the algorithm (rather than the true value $f(\theta)$).
From the concavity of $\sqrt{\cdot}$, if $\t{f}(\theta)$ is larger then $\|\t{x}_i(\theta)-\hat{x}_i(\theta)\|$ must be smaller to achieve the same $\delta_f$.

Lastly, we note the key reason why we require \eqref{EQ:MIN_ACCURACY}: it guarantees that our estimate $\t{\rho}^k$ of $\rho^k$ is not too inaccurate.

\begin{lemma} \label{lem_rho_relationships}
	Suppose $|\t{f}(\theta^k)-f(\theta^k)|\leq\delta^k$ and $|\t{f}(\theta^k+s^k)-f(\theta^k+s^k)|\leq\delta^k_+$.
	If \eqref{EQ:MIN_ACCURACY} holds, then $|\t{\rho}^k-\rho^k| \leq 2\eta_1^{\prime}$.
\end{lemma}
\begin{proof}
	Follows immediately from \eqref{EQ:RHO_NOISY} and \eqref{EQ:RHO_TRUE}; see \cite[Section 10.6.1]{Conn2000}.
\end{proof}

\subsection{Convergence and Worst-Case Complexity}
We now prove the global convergence of \algref{alg_dynamic_accuracy_dfo} and analyse its worst-case complexity (i.e.~the number of iterations required to achieve $\|\nabla f(\theta^k)\|\leq\epsilon$ for the first time).

\begin{assumption} \label{ass_cauchy_bdd_hess}
	The computed trust-region step $s^k$ satisfies
	\begin{align}
		m^k(0)-m^k(s^k) \geq \frac{1}{2} \|g^k\| \min\left(\Delta^k, \frac{\|g^k\|}{\|H^k\|+1}\right), \label{eq_cauchy}
	\end{align}
	and there exists $\kappa_H\geq 1$ such that $\|H^k\| + 1 \leq \kappa_H$ for all $k$.
\end{assumption}

\assref{ass_cauchy_bdd_hess} is standard and the condition \eqref{eq_cauchy} easy to achieve in practice \cite[Chapter 6.3]{Conn2000}.

Firstly, we must show that the inner loops for the criticality and accuracy phases terminate.
We begin with the criticality phase, and then consider the accuracy phase.

\begin{lemma}[{\cite[Lemma B.1]{Cartis2019a}}] \label{lem_criticality}
	Suppose \assref{ASS:SMOOTHNESS} holds and $\|\nabla f(\theta^k)\|\geq\epsilon>0$. 
	Then the criticality phase terminates in finite time with
	\begin{align}
		\min\left(\Delta^k_{\rm init}, \frac{\gamma_{\rm dec}\epsilon}{\kappaeg+1}\right) \leq \Delta^k \leq \Delta^k_{\rm init},
	\end{align}
	where $\Delta^k_{\rm init}$ is the value of $\Delta^k$ before the criticality phase begins.
\end{lemma}

\begin{lemma}[{\cite[Lemma 3.7]{Cartis2019a}}] \label{lem_eps_g}
	Suppose \assref{ASS:SMOOTHNESS} holds.
	Then in all iterations we have $\|g^k\| \geq \min(\epsilon, \Delta^k)$. 
	Also, if $\|\nabla f(\theta^k)\|\geq\epsilon>0$, then $\|g^k\| \geq \epsilon/(\kappaeg + 1) > 0$.
\end{lemma}

We note that our presentation of the criticality phase here can be made more general by allowing $\|g^k\|\geq \epsilon_C \neq \epsilon$ as the entry test, setting $\Delta^k$ to $\omega^i \Delta^k$ for some $\omega\in(0,1)$ possibly different to $\gamma_{\rm dec}$, and having an exit test $\Delta^k \leq \mu\|g^k\|$ for some $\mu>0$.
All the below results hold under these assumptions, with modifications as per \cite{Cartis2019a}.

\begin{lemma}
	If Assumptions \ref{ASS:SMOOTHNESS} and \ref{ass_cauchy_bdd_hess} hold and $\|\nabla f(\theta^k)\|\geq\epsilon>0$, then the accuracy phase terminates in finite time (i.e.~line \ref{ln_accuracy_end} of \algref{alg_dynamic_accuracy_dfo} is eventually called)
\end{lemma}
\begin{proof}
	From \lemref{lem_eps_g} we have $\|g^k\|\geq\epsilon/(\kappaeg + 1)$, and the result then follows from \cite[Lemma 10.6.1]{Conn2000}.
\end{proof}

We now collect some key preliminary results required to establish complexity bounds.

\begin{lemma} \label{eq_min_success}
	Suppose Assumptions \ref{ASS:SMOOTHNESS} and \ref{ass_cauchy_bdd_hess} hold, $m^k$ is fully linear in $B(\theta^k,\Delta^k)$ and
	\begin{align}
		\Delta^k &\leq c_0 \|g^k\|, \nonumber \\
		&\text{where} \quad c_0 := \min\left(\frac{1-\eta_2-2\eta_1^{\prime}}{4\kappa_{\rm ef}}, \frac{1}{\kappa_H}\right) > 0, \label{eq_c0}
	\end{align}
	then $\t{\rho}^k\geq\eta_2$.
\end{lemma}
\begin{proof}
	We compute
	\begin{align}
		|\rho^k-1| &= \left|\frac{(f(\theta^k)-f(\theta^k+s^k))-(m^k(0)-m^k(s^k))}{m^k(0)-m^k(s^k)}\right|, \\
		&\leq \frac{|f(\theta^k+s^k)-m^k(s^k)|}{|m^k(0)-m^k(s^k)|} + \frac{|f(\theta^k)-m^k(0)|}{|m^k(0)-m^k(s^k)|}.
	\end{align}
	Since $\Delta^k \leq \|g^k\|/\kappa_H$, from \assref{ass_cauchy_bdd_hess} we have
	\begin{align}
		m^k(0)-m^k(s^k) \geq \frac{1}{2}\|g^k\|\Delta^k.
	\end{align}
	From this and full linearity, we get
	\begin{align}
		|\rho^k-1| \leq 2\left(\frac{2\kappa_{\rm ef}(\Delta^k)^2}{\|g^k\|\Delta^k}\right) \leq 1-\eta_2-2\eta_1^{\prime},
	\end{align}
	and so $\rho^k\geq\eta_2+2\eta_1^{\prime}$, hence $\t{\rho}^k\geq\eta_2$ from \lemref{lem_rho_relationships}.
\end{proof}

\begin{lemma} \label{lem_delta_min}
	Suppose Assumptions \ref{ASS:SMOOTHNESS} and \ref{ass_cauchy_bdd_hess} hold.
	Suppose $\|\nabla f(\theta^k)\|\geq\epsilon$ for all $k=0,\ldots,k_{\epsilon}$ and some $\epsilon\in(0,1)$.
	Then, for all $k\leq k_{\epsilon}$,
	\begin{align}
		\Delta^k \geq \Delta_{\min} := \gamma_{\rm dec}\min\left(\Delta^0, \frac{c_0\epsilon}{\kappaeg+1}, \frac{\gamma_{\rm dec}\epsilon}{\kappaeg+1}\right) > 0. \label{eq_delta_min_dfo}
	\end{align}
\end{lemma}
\begin{proof}
	As above, we let $\Delta^k_{\rm init}$ and $m^k_{\rm init}$ denote the values of $\Delta^k$ and $m^k$ before the criticality phase (i.e.~$\Delta^k_{\rm init}=\Delta^k$ and $m^k_{\rm init}=m^k$ if the criticality phase is not called).
	From \lemref{lem_eps_g}, we know $\|g^k\|\geq\epsilon/(\kappaeg+1)$ for all $k\leq k_{\epsilon}$.
	Suppose by contradiction $k\leq k_{\epsilon}$ is the first iteration such that $\Delta^k<\Delta_{\min}$.
	Then from \lemref{lem_criticality},
	\begin{align}
		\frac{\gamma_{\rm dec}\epsilon}{\kappaeg+1} \geq \Delta_{\min} > \Delta^k \geq \min\left(\Delta^k_{\rm init}, \frac{\gamma_{\rm dec}\epsilon}{\kappaeg+1}\right),
	\end{align}
	and so $\Delta^k\geq \Delta^k_{\rm init}$; hence $\Delta^k=\Delta^k_{\rm init}$.
	That is, either the criticality phase is not called, or terminates with $i=0$ (in this case, the model $m^k$ is formed simply by making $m^k_{\rm init}$ fully linear in $B(\theta^k,\Delta^k)=B(\theta^k,\Delta^k_{\rm init})$).
	
	If the accuracy phase loop occurs, we go back to the criticality phase, which can potentially happen multiple times.
	However, since the only change is that $\t{r}(\theta^k)$ is evaluated to higher accuracy, incorporating this information into the model $m^k$ can never destroy full linearity. 
	Hence, after the accuracy phase, by the same reasoning as above, either one iteration of the criticality phase occurs (i.e.~$m^k$ is made fully linear) or it is not called.
	If the accuracy phase is called multiple times and the criticality phase occurs multiple times, all times except the first have no effect (since the accuracy phase can never destroy full linearity).
	Thus $\Delta^k$ is unchanged by the accuracy phase.
	
	Since $\Delta_{\min}<\Delta^0_{\rm init}$, we have $k\geq 1$.
	As $k$ is the first iteration such that $\Delta^k<\Delta_{\min}$ and $\Delta^k=\Delta^k_{\rm init}$, we must have $\Delta^k_{\rm init}=\gamma_{\rm dec}\Delta^{k-1}$ (as this is the only other way $\Delta^k$ can be reduced).
	Therefore $\Delta^{k-1}=\Delta^k/\gamma_{\rm dec}<\Delta_{\min}/\gamma_{\rm dec}$, and so
	\begin{align}
		\Delta^{k-1} \leq \min\left(\frac{c_0\epsilon}{\kappaeg+1}, \frac{\gamma_{\rm dec}\epsilon}{\kappaeg+1}\right). \label{eq_tmp1}
	\end{align}
	We then have $\Delta^{k-1} \leq c_0\epsilon/(\kappaeg+1) \leq c_0\|g^{k-1}\|$, and so by \eqref{eq_min_success} either $\t{\rho}^k\geq\eta_2$ or $m^{k-1}$ is not fully linear.
	Either way, we set $\Delta^k_{\rm init}\geq\Delta^{k-1}$ in \eqref{eq_delta_update}.
	This contradicts $\Delta^k_{\rm init}=\gamma_{\rm dec}\Delta^{k-1}$ above, and we are done.
\end{proof}

We now bound the number of iterations of each type.
Specifically, we suppose that $k_{\epsilon}+1$ is the first $k$ such that $\|\nabla f(\theta^k)\|\geq\epsilon$.
Then, we define the sets of iterations:
\begin{itemize}
	\item $\mathcal{S}_\epsilon$ is the set of iterations $k\in\{0,\ldots,k_{\epsilon}\}$ which are `successful'; i.e.~$\t{\rho}^k\geq\eta_2$, or $\t{\rho}^k\geq\eta_1$ and $m^k$ is fully linear in $B(\theta^k,\Delta^k)$.
	\item $\mathcal{M}_\epsilon$ is the set of iterations $k\in\{0,\ldots,k_{\epsilon}\}$ which are `model-improving'; i.e.~$\t{\rho}^k<\eta_2$ and $m^k$ is not fully linear in $B(\theta^k,\Delta^k)$.
	\item $\mathcal{U}_\epsilon$ is the set of iterations $k\in\{0,\ldots,k_{\epsilon}\}$ which are `unsuccessful'; i.e.~$\t{\rho}^k<\eta_1$ and $m^k$ is fully linear in $B(\theta^k,\Delta^k)$.
\end{itemize}
These three sets form a partition of $\{0,\ldots,k_{\epsilon}\}$.

\begin{proposition} \label{prop_succ_complexity}
	Suppose Assumptions \ref{ASS:SMOOTHNESS} and \ref{ass_cauchy_bdd_hess} hold.
	Then
	\begin{align}
		|\mathcal{S}_{\epsilon}| \leq \frac{2 (\kappaeg+1) f(\theta^0)}{(\eta_1-2\eta_1^{\prime})\epsilon\Delta_{\min}}. \label{eq_successful_count}
	\end{align}
\end{proposition}
\begin{proof}
	By definition of $k_{\epsilon}$, $\|\nabla f(\theta^k)\|\geq\epsilon$ for all $k\leq k_{\epsilon}$ and so \lemref{lem_eps_g} and \lemref{lem_delta_min} give $\|g^k\|\geq\epsilon/(\kappaeg+1)$ and $\Delta^k\geq\Delta_{\min}$ for all $k\leq k_{\epsilon}$ respectively.
	For any $k\leq k_{\epsilon}$ we have
	\begin{align}
		f(\theta^k) &- f(\theta^{k+1}) \nonumber \\
		&= \rho^k [m^k(0)-m^k(s^k)], \\
		&\geq \frac{1}{2}\rho^k \|g^k\|\min\left(\Delta^k, \frac{\|g^k\|}{\|H^k\|+1}\right), \\
		&\geq \frac{1}{2}\rho^k \frac{\epsilon}{\kappaeg+1} \min\left(\Delta_{\min}, \frac{\epsilon}{\kappa_H(\kappaeg+1)}\right),
	\end{align}
	by definition of $\rho^k$ and \assref{ass_cauchy_bdd_hess}.
	If $k\in\mathcal{S}_{\epsilon}$, we know $\t{\rho_k}\geq\eta_1$, which implies $\rho^k\geq\eta_1-2\eta_1^{\prime}>0$ from \lemref{lem_rho_relationships}.
	Therefore
	\begin{align}
		f(\theta^k) &- f(\theta^{k+1}) \nonumber \\
		&\geq \frac{1}{2}(\eta_1-2\eta_1^{\prime}) \frac{\epsilon}{\kappaeg+1} \min\left(\Delta_{\min}, \frac{\epsilon}{\kappa_H(\kappaeg+1)}\right), \\
		&=  \frac{1}{2}(\eta_1-2\eta_1^{\prime}) \frac{\epsilon}{\kappaeg+1} \Delta_{\min},
	\end{align}
	for all $k\in\mathcal{S}_{\epsilon}$, where the last line follows since $\Delta_{\min} < c_0\epsilon/(\kappaeg+1) \leq \epsilon/[\kappa_H(\kappaeg+1)]$ by definition of $\Delta_{\min}$ \eqref{eq_delta_min_dfo} and $c_0$ \eqref{eq_c0}.
	
	The iterate $\theta^k$ is only changed on successful iterations (i.e.~$\theta^{k+1}=\theta^k$ for all $k\notin\mathcal{S}_{\epsilon}$).
	Thus, as $f(\theta)\geq 0$ from the least-squares structure \eqref{eq_upper_exact}, we get
	\begin{align}
		f(\theta^0) &\geq f(\theta^0) - f(\theta^{k_{\epsilon}+1}), \\
		&= \sum_{k\in\mathcal{S}_{\epsilon}} f(\theta^k) - f(\theta^{k+1}), \\
		&\geq |\mathcal{S}_{\epsilon}| \left[\frac{1}{2}(\eta_1-2\eta_1^{\prime}) \frac{\epsilon}{\kappaeg+1} \Delta_{\min}\right],
	\end{align}
	and the result follows.
\end{proof}

We are now in a position to prove our main results.

\begin{proof}[Proof of \thmref{thm_iter_complexity}]
	To derive a contradiction, suppose that $\|\nabla f(\theta^k)\|\geq\epsilon$ for all $k\in\{0,\ldots,K\}$, and so $\|g^k\|\geq\epsilon/(\kappaeg+1)$ and $\Delta^k\geq\Delta_{\min}$ by \lemref{lem_eps_g} and \lemref{lem_delta_min} respectively.
	Since $K\leq k_{\epsilon}$ by definition of $k_{\epsilon}$, we will try to construct an upper bound on $k_{\epsilon}$.
	We already have an upper bound on $|\mathcal{S}_{\epsilon}|$ from \propref{prop_succ_complexity}.

	If $k\in\mathcal{S}_{\epsilon}$, we set $\Delta^{k+1}\leq \gamma_{\rm inc}\Delta^k$.
	Similarly, if $k\in\mathcal{U}_{\epsilon}$ we set $\Delta^{k+1}=\gamma_{\rm dec}\Delta^k$.
	Thus
	\begin{align}
		\Delta_{\min} \leq \Delta^{k_{\epsilon}} \leq \Delta^0 \gamma_{\rm inc}^{|\mathcal{S}_{\epsilon}|} \gamma_{\rm dec}^{|\mathcal{U}_{\epsilon}|}.
	\end{align}
	That is, $\Delta_{\min}/\Delta^0 \leq \gamma_{\rm inc}^{|\mathcal{S}_{\epsilon}|} \gamma_{\rm dec}^{|\mathcal{U}_{\epsilon}|}$, and so 
	\begin{align}
		|\mathcal{U}_{\epsilon}| &\leq \frac{\log\gamma_{\rm inc}}{|\log \gamma_{\rm dec}|}|\mathcal{S}_{\epsilon}| + \frac{\log(\Delta^0/\Delta_{\min})}{|\log \gamma_{\rm dec}|}, \label{eq_unsuccessful_count}
	\end{align}
	noting we have changed $\Delta_{\min}/\Delta^0<1$ to $\Delta^0/\Delta_{\min}>1$ and used $\log \gamma_{\rm dec} < 0$, so all terms in \eqref{eq_unsuccessful_count} are positive.
	Now, the next iteration after a model-improving iteration cannot be model-improving (as the resulting model is fully linear), giving
	\begin{align}
		|\mathcal{M}_{\epsilon}| &\leq |\mathcal{S}_{\epsilon}| + |\mathcal{U}_{\epsilon}|. \label{eq_model_improving_count}
	\end{align}
	If we combine \eqref{eq_unsuccessful_count} and \eqref{eq_model_improving_count} with $k_{\epsilon}\leq |\mathcal{S}_{\epsilon}| + |\mathcal{M}_{\epsilon}| + |\mathcal{U}_{\epsilon}|$, we get
	\begin{align}
		k_{\epsilon} &\leq 2(|\mathcal{S}_{\epsilon}| + |\mathcal{U}_{\epsilon}|), \\
		&\leq \left(2 + 2\frac{\log\gamma_{\rm inc}}{|\log \gamma_{\rm dec}|}\right)|\mathcal{S}_{\epsilon}| + 2\frac{\log(\Delta^0/\Delta_{\min})}{|\log \gamma_{\rm dec}|},
	\end{align}
	which, given the bound on $|\mathcal{S}_{\epsilon}|$ \eqref{eq_successful_count} means $K\leq k_{\epsilon}$ is bounded above by the right-hand side of \eqref{eq_kepsilon}, a contradiction.
\end{proof}

\begin{proof}[Proof of \corref{cor_complexity}]
	\revision{The iteration bound follows directly from \thmref{thm_iter_complexity}, noting that $\Delta_{\min}=\bigO(\kappa^{-2} \epsilon)$. This also implies that $\liminf_{k\to\infty} \|\nabla f(\theta^k)\|=0$ and so \eqref{eq_liminf} holds from the same argument as in \cite[Theorem 10.13]{Conn2009} without modification.}
	
	For the evaluation bound, we also need to count the number of inner iterations of the criticality phase.
	Suppose $\|\nabla f(\theta^k)\|<\epsilon$ for $k=0,\ldots,k_{\epsilon}$.
	Similar to the above, we define: (a) $\mathcal{C}^{M}_{\epsilon}$ to be the number of criticality phase iterations corresponding to the first iteration of $i=0$ where $m^k$ whas not already fully linear, in iterations $0,\ldots,k_{\epsilon}$; and (b) $\mathcal{C}^{U}_{\epsilon}$ to be the number of criticality phase iterations corresponding to all other iterations $i>0$ (where $\Delta^k$ is reduced and $m^k$ is made fully linear) in iterations $0,\ldots,k_{\epsilon}$.
	
	From \lemref{lem_delta_min} we have $\Delta^k\geq\Delta_{\min}$ for all $k\leq k_{\epsilon}$.
	We note that $\Delta^k$ is reduced by a factor $\gamma_{\rm dec}$ for every iteration of the criticality phase in $\mathcal{C}^{U}_{\epsilon}$.
	Thus by a more careful reasoning as we used to reach \eqref{eq_unsuccessful_count}, we conclude 
	\begin{align}
		\Delta_{\min} &\leq \Delta^0 \gamma_{\rm inc}^{|\mathcal{S}_{\epsilon}|} \gamma_{\rm dec}^{|\mathcal{U}_{\epsilon}| + |\mathcal{C}^{U}_{\epsilon}|}, \\
		|\mathcal{C}^{U}_{\epsilon}| &\leq \frac{\log\gamma_{\rm inc}}{|\log \gamma_{\rm dec}|}|\mathcal{S}_{\epsilon}| + \frac{\log(\Delta^0/\Delta_{\min})}{|\log \gamma_{\rm dec}|} - |\mathcal{U}_{\epsilon}|. \label{eq_criticality_count2}
	\end{align}
	Also, after every iteration $k$ in which the first iteration of criticality phase makes $m^k$ fully linear, we have either a (very) successful or unsuccessful step, not a model-improving step.
	From the same reasoning as in \lemref{lem_delta_min}, the accuracy phase can only cause at most one more step criticality phase in which $m^k$ is made fully linear, regardless of how many times it is called.\footnote{Of course, there may be many more initial steps of the criticality phase in which $m^k$ is already fully linear, but no work is required in this case.}
	Thus,
	\begin{align}
		|\mathcal{C}^{M}_{\epsilon}| &\leq 2\left(|\mathcal{S}_{\epsilon}| + |\mathcal{U}_{\epsilon}|\right). \label{eq_criticality_count}
	\end{align}
	Combining \eqref{eq_criticality_count2} and \eqref{eq_criticality_count} with \eqref{eq_unsuccessful_count} and \eqref{eq_model_improving_count}, we conclude that the number of times we make $m^k$ fully linear is 
	\begin{align}
		|\mathcal{M}_{\epsilon}| &+ |\mathcal{C}^{U}_{\epsilon}| + |\mathcal{C}^{M}_{\epsilon}|  \nonumber \\
		&\leq \left(3 + 3\frac{\log\gamma_{\rm inc}}{|\log \gamma_{\rm dec}|}\right)|\mathcal{S}_{\epsilon}| + 3\frac{\log(\Delta^0/\Delta_{\min})}{|\log \gamma_{\rm dec}|}, \\
		&\leq \left(3 + 3\frac{\log\gamma_{\rm inc}}{|\log \gamma_{\rm dec}|}\right)\frac{2 (\kappaeg+1) f(\theta^0)}{(\eta_1-2\eta_1^{\prime})\epsilon\Delta_{\min}} \nonumber \\
		&\qquad\qquad\qquad\qquad + 3\frac{\log(\Delta^0/\Delta_{\min})}{|\log \gamma_{\rm dec}|},
	\end{align}
	where the second inequality follows from \propref{prop_succ_complexity}.
	
	If $\epsilon<1$, we conclude that the number of times we make $m^k$ fully linear before $\|\nabla f(\theta^k)\|<\epsilon$ for the first time is the same as the number of iterations, $\bigO(\kappa^3\epsilon^{-2})$.
	Since each iteration requires one new objective evaluation (at $\theta^k+s^k$) and each time we make $m^k$ fully linear requires at most $\bigO(d)$ objective evaluations (corresponding to replacing the entire interpolation set), we get the stated evaluation complexity bound.
\end{proof}

\subsection{Estimating the Lower-Level Work}
We have from \corref{cor_complexity} that we can achieve $\|\nabla f(\theta^k)\|<\epsilon$ in $\bigO(\epsilon^{-2})$ evaluations of $\t{r}(\theta)$.
In this section, we use the fact that evaluations of $\t{r}(\theta)$ come from finitely terminating a linearly-convergent procedure (i.e.~strongly convex optimization) to estimate the total work required in the lower-level problem.
This is particularly relevant in an imaging context, where the lower-level problem can be large-scale and poorly-conditioned; this can be the dominant cost of \algref{alg_dynamic_accuracy_dfo}.

\begin{proposition} \label{prop_lower_level_work}
	Suppose Assumptions \ref{ASS:SMOOTHNESS} and \ref{ass_cauchy_bdd_hess} hold and $\|\nabla f(\theta^k)\|\geq\epsilon$ for all $k=0,\ldots,k_{\epsilon}$ and some $\epsilon\in(0,1]$.
	Then for every objective evaluation in iterations $k\leq k_{\epsilon}$ it suffices to guarantee that $\|\t{x}_i(\theta)-\hat{x}_i(\theta)\| = \bigO(\epsilon^2)$ for all $i=1,\ldots,n$.
\end{proposition}
\begin{proof}
	For all $k\leq k_{\epsilon}$ we have $\|g^k\|\geq\epsilon/(\kappaeg+1)$ and $\Delta^k\geq\Delta_{\min}$ by \lemref{lem_eps_g} and \lemref{lem_delta_min} respectively.
	There are two places where we require upper bounds on $\|\t{x}_i(\theta)-\hat{x}_i(\theta)\|$ in our objective evaluations: ensuring $\t{f}(\theta^k)$ and $\t{f}(\theta^k+s^k)$ satisfy \eqref{EQ:MIN_ACCURACY} and ensuring our model is fully linear using \lemref{lem_fully_linear}.
	
	In the first case, we note that
	\begin{align}
		m^k(0) - m^k(s^k) &\geq \frac{1}{2} \frac{\epsilon}{\kappaeg+1} \min\left(\Delta_{\min}, \frac{\epsilon}{\kappa_H(\kappaeg+1)}\right), \\
		&= \frac{1}{2} \frac{\epsilon}{\kappaeg+1}\Delta_{\min},
	\end{align}
	by \assref{ass_cauchy_bdd_hess} and using $\Delta_{\min} < c_0\epsilon/(\kappaeg+1) \leq \epsilon/[\kappa_H(\kappaeg+1)]$ by definition of $\Delta_{\min}$ \eqref{eq_delta_min_dfo} and $c_0$ \eqref{eq_c0}.
	Therefore to ensure \eqref{EQ:MIN_ACCURACY} it suffices to guarantee
	\begin{align}
		\max &\left(|\t{f}(\theta^k) - f(\theta^k)|, |\t{f}(\theta^k+s^k) - f(\theta^k+s^k)|\right) \nonumber \\
		&\leq \delta_f^{\min} := \frac{1}{2} \eta_1^{\prime} \frac{\epsilon}{\kappaeg+1}\Delta_{\min}.
	\end{align}
	From \lemref{lem_inexact_eval_accuracy}, specifically the second part of \eqref{EQ:DELTA_BOUND1}, this means to achieve \eqref{EQ:MIN_ACCURACY} it suffices to guarantee
	\begin{align}
		\|\t{x}_i(\theta)-\hat{x}_i(\theta)\| \leq \sqrt{f(\theta)+\delta_f^{\min}} - \sqrt{f(\theta)},
	\end{align}
	for all $i=1,\ldots,n$, where $\theta\in\cup_{k\leq k_{\epsilon}} \{\theta^k, \theta^k+s^k\}$.
	From \assref{ASS:SMOOTHNESS} we have $f(\theta)\leq f_{\max} := r_{\max}^2/n$, and so from the fundamental theorem of calculus we have
	\begin{align}
		\sqrt{f(\theta)+\delta_f^{\min}} &- \sqrt{f(\theta)} = \int_{f(\theta)}^{f(\theta)+\delta_f^{\min}} \frac{1}{2\sqrt{t}} dt, \\
		&\geq \frac{\delta_f^{\min}}{2\sqrt{f(\theta)+\delta_f^{\min}}} \geq \frac{\delta_f^{\min}}{2\sqrt{f_{\max}+\delta_f^{\min}}}.
	\end{align}
	Since $\epsilon<1$, $\delta_f^{\min}$ is bounded above by a constant and so $\sqrt{f_{\max}+\delta_f^{\min}}$ is bounded above.
	Thus \eqref{EQ:MIN_ACCURACY} is achieved provided $\|\t{x}_i(\theta)-\hat{x}_i(\theta)\| = \bigO(\delta_f^{\min})$ for all $i=1,\ldots,n$.
	
	For the second case (ensuring full linearity), we need to guarantee \eqref{eq_desired_accuracy} holds.
	This is achieved provided $\|\t{x}_i(\theta)-\hat{x}_i(\theta)\| = \bigO(\Delta_{\min}^2)$ for all $i=1,\ldots,n$.
	The result then follows by noting $\delta_f^{\min}=\bigO(\epsilon\Delta_{\min})$ and $\Delta_{\min}=\bigO(\epsilon)$.
\end{proof}

\corref{cor_complexity} and \propref{prop_lower_level_work} say that to ensure $\|\nabla f(\theta^k)\|<\epsilon$ for some $k$, we have to perform $\bigO(d\kappa^3\epsilon^{-2})$ upper-level objective evaluations, each requiring accuracy at most  $\|\t{x}_i(\theta)-\hat{x}_i(\theta)\| = \bigO(\epsilon^2)$ for all $i$.
Since our lower-level evaluations correspond to using GD/FISTA to solve a strongly convex problem, the computational cost of each upper-level evaluation is $\bigO(n\log(\epsilon^{-2}))$ provided we have reasonable initial iterates.
From this, we conclude that the total computational cost before achieving $\|\nabla f(\theta^k)\| <\epsilon$ is at most $\bigO(\epsilon^{-2}\log(\epsilon^{-1}))$ iterations of the lower-level algorithm.
However, this is a conservative approach to estimating the cost: many of the iterations correspond to $\|\nabla f(\theta^k)\|\gg\epsilon$, and so the work required for these is less.
This suggests the question: \emph{can we more carefully estimate the work required at different accuracy levels to prove a lower $\epsilon$-dependence on the total work?}
We now argue that this is not possible without further information about asymptotic convergence rates (e.g.~local convergence theory).
For simplicity we drop all constants and $\bigO(\cdot)$ notation in the below.

Suppose we count the work required to achieve progressively higher accuracy levels $1\geq\epsilon_0>\epsilon_1>\cdots>\epsilon_N:=\epsilon$ for some desired accuracy $\epsilon\ll 1$.
Since each $\epsilon_i<1$, we assume that we require $\epsilon_i^{-2}$ evaluations to achieve accuracy $\epsilon_i$, where each evaluation requires $\log(\epsilon_i^{-1})$ computational work.
We may choose $\epsilon_0<1$, since the cost to achieve accuracy $\epsilon_0$ is fixed (i.e.~independent of our desired accuracy $\epsilon$), so does not affect our asymptotic bounds.
Counting the total lower-level problem work---which we denote $W(\epsilon)$---in this way, we get
\begin{align}
	W(\epsilon) = W(\epsilon_0) + \sum_{i=1}^{N}\left(\epsilon_i^{-2}-\epsilon_{i-1}^{-2}\right)\log(\epsilon_i^{-1}). \label{eq_work_approx}
\end{align}
The second term of \eqref{eq_work_approx} corresponds to a right Riemann sum approximating $\int_{\epsilon_0^{-2}}^{\epsilon_N^{-2}}\log(\sqrt{x})dx$.
Since $x\to \log(\sqrt{x})=\log(x)/2$ is strictly increasing, the right Riemann sum overestimates the integral; hence
\begin{align}
	W(\epsilon) &\geq W(\epsilon_0) + \frac{1}{2}\int_{\epsilon_0^{-2}}^{\epsilon^{-2}}\log(x) dx, \\
	&= W(\epsilon_0) + \frac{1}{2}\left[\epsilon^{-2}(\log(\epsilon^{-2}) -1) - \epsilon_0^{-2}(\log(\epsilon_0^{-2})-1)\right],
\end{align}
independent of our choices of $\epsilon_1,\ldots,\epsilon_{N-1}$.
That is, as $\epsilon\to 0$, we have $W(\epsilon)\sim\epsilon^{-2}\log(\epsilon^{-1})$, so our na\"ive estimate is tight.

We further note that this na\"ive bound applies more generally.
Suppose the work required for a single evaluation of the lower-level objective to accuracy $\epsilon$ is $w(\epsilon^{-2})\geq 0$ (e.g.~$w(x)=\log(x)/2$ above).
Assuming $w$ is increasing (i.e.~higher accuracy evaluations require more work), we get, similarly to above,
\begin{align}
    W(\epsilon) \geq W(\epsilon_0) +  \int_{\epsilon_0^{-2}}^{\epsilon^{-2}} w(x) dx.
\end{align}
Since $w$ is increasing and nonnegative, by 
\begin{align}
    \int_{\epsilon_0^{-2}}^{\epsilon^{-2}} w(x) dx &\geq \int_{(\epsilon_0^{-2}+\epsilon^{-2})/2}^{\epsilon^{-2}} w(x) dx, \\
    &\geq \frac{\epsilon_0^{-2}+\epsilon^{-2}}{2}w\left(\frac{\epsilon_0^{-2}+\epsilon^{-2}}{2}\right),
\end{align}
the na\"ive work bound $W(\epsilon)\sim \epsilon^{-2}w(\epsilon^{-2})$ holds provided $w(x) = \bigO(w(x/2))$ as $x\to \infty$; that is, $w(x)$ does not increase too quickly.
This holds in a variety of cases, such as $w(x)$ bounded, concave or polynomial (but not if $w(x)$ grows exponentially).
In particular, this holds for $w(x)\sim\log(x)/2$ as above, and $w(x)\sim x^{1/2}$ and $w(x)\sim x$, which correspond to the work required (via standard sublinear complexity bounds) if the lower-level problem is a strongly convex, convex or nonconvex optimization problem respectively.

\section{Numerical Results} \label{sec_numerics}

\subsection{Upper-level solver (DFO-LS)}
We implement the dynamic accuracy algorithm (\algref{alg_dynamic_accuracy_dfo}) in DFO-LS \cite{Cartis2018}, an open-source Python package which solves nonlinear least-squares problems subject to bound constraints using model-based DFO.\footnote{Available at \url{https://github.com/numericalalgorithmsgroup/dfols}.}
As described in \cite{Cartis2018}, DFO-LS has a number of modifications compared to the theoretical algorithm \algref{alg_dynamic_accuracy_dfo}.
The most notable modifications here are that DFO-LS:
\begin{itemize}
    \item Allows for bound constraints (and internally scales variables so that the feasible region is $[0,1]$ for all variables);
    \item Does not implement a criticality phase;
    \item Uses a simplified model-improving step;
    \item Maintains two trust-region radii to avoid decreasing $\Delta^k$ too quickly;
    \item Implements a `safety phase', which treats iterations with short steps $\|s^k\|\ll\Delta^k$ similarly to unsuccessful iterations.
\end{itemize}
More discussion on DFO-LS can be found~\cite{Cartis2019a,Cartis2018}.

Here, we use DFO-LS v1.1.1, modified for the dynamic accuracy framework as described above.
When determining the accuracy level for a given evaluation, we require accuracy level $\delta_x=10(\Delta^k)^2$ for all evaluations (c.f.~\lemref{lem_fully_linear}), and also \eqref{EQ:MIN_ACCURACY} when checking objective decrease \eqref{EQ:RHO_NOISY}.

\subsection{Application: 1D Image Denoising} \label{sec_denoising}
In this section, we consider the application of DFO-LS to the problem of learning the regularization and smoothing parameters for the image denoising model \eqref{EQ:LOWER:APPROX:TV} as described in \secref{sec_tv_denoising}.
We use training data constructed using the method described in \secref{sec_data_model} with $N=256$ and $\sigma=0.1$.

\paragraph{1-parameter case}
The simplest example we consider is the 1-parameter case, where we only wish to learn $\alpha$ in \eqref{EQ:LOWER:APPROX:TV}. We fix $\nu = \xi =10^{-3}$ and use a training set of $n=10$ randomly-generated images.
We choose $\alpha = 10^\theta$, optimize over $\theta$ within bounds $\theta \in[-7,7]$ with starting value $\theta^0=0$.
We do not regularize this problem, i.e.~$\mathcal{J}=0$.

\paragraph{3-parameter case}
We also consider the more complex problem of learning three parameters for the denoising problem (namely $\alpha$, $\nu$ and $\xi$).
We choose to penalize a large condition number of the lower-level problem, thus promotes efficient solution of the lower-level problem after training. To be precise we choose
\begin{align}
\mathcal{J}(\alpha,\nu,\xi) = \left(\frac{L(\alpha,\nu,\xi)}{\mu(\alpha,\nu,\xi)}\right)^2
\end{align}
where $L$ and $\mu$ are the smoothness and strong convexity constants given in \secref{sec_tv_denoising}.

The problem is solved using the parametrization $\alpha = 10^{\theta_1}, \nu = 10^{\theta_2}$ and $\xi = 10^{\theta_3}$.
Here, we use a training set of $n=20$ randomly-generated images, and optimize over $\theta \in [-7,7] \times [-7,0]^2$.
Our default starting value is $\theta^0 = (0, -1, -1)$ and our default choice of upper-level regularization parameter is $\beta=10^{-6}$.

\paragraph{Solver settings}
We run DFO-LS with a budget of 20 and 100 evaluations of the upper-level objective $f$ for the 1- and 3-parameter cases respectively, and with $\rho_{\rm end}=10^{-6}$ in both cases.
We compare the dynamic accuracy variant of DFO-LS (given by \algref{alg_dynamic_accuracy_dfo}) against two variants of DFO-LS (as originally implemented in \cite{Cartis2018}):
\begin{enumerate}
    \item Low-accuracy evaluations: each value $\hat x_i$ received by DFO-LS is inaccurately estimated via a fixed number of iterations of GD/FISTA; we use 1,000 iterations of GD and 200 iterations of FISTA.
    \item High-accuracy evaluations: each value $z_i$ received by DFO-LS is estimated using 10,000 iterations of GD or \revision{2,000} iterations of FISTA.
\end{enumerate}
We estimate $\delta_f$ in the plots below by taking $\delta_r$ to be the maximum estimate of $\|\hat x_i(\theta)-x_i\|$ for each $i=1,\ldots,n$.
When running the lower-level solvers, our starting point is the final reconstruction from the previous upper-level evaluation, which we hope is a good estimate of the solution.

\paragraph{1-parameter denoising results}
In \figref{fig_denoising1_comparison} we compare the six algorithm variants (low, high and dynamic accuracy versions of both GD and FISTA) on the 1-parameter denoising problem.
Firstly in Figures \ref{fig_denoising1_comparison_objective} and \ref{fig_denoising1_comparison_objective_high_acc}, we show the best upper-level objective value observed against `computational cost', measured as the total GD/FISTA iterations performed (over all upper-level evaluations).
For each variant, we plot the value $\t{f}(\theta)$ and the uncertainty range $\t{f}(\theta)\pm\delta_f$ associated with that evaluation.
In \figref{fig_denoising1_comparison_param} we show the best $\alpha_{\theta}$ found against the same measure of computational cost.

\begin{figure}
    \centering
    \subfloat[Objective value $f(\theta)$]{\label{fig_denoising1_comparison_objective}\includegraphics[width=6cm, height=\HeightGr]{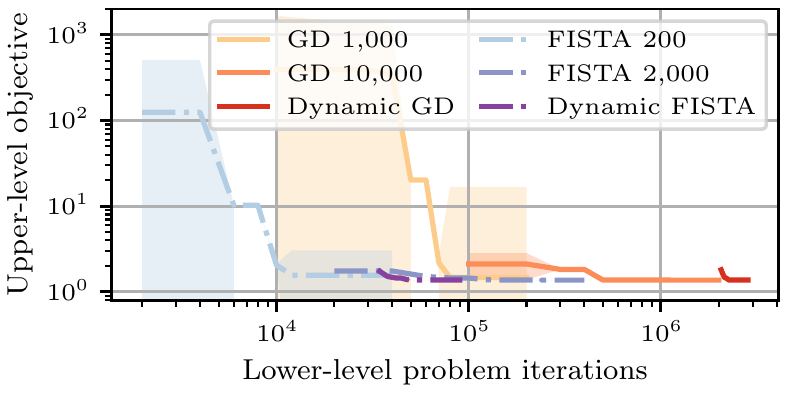}}
    \hspace{1cm}
    \subfloat[Objective value $f(\theta)$, zoomed in]{\label{fig_denoising1_comparison_objective_high_acc}\includegraphics[width=6cm, height=\HeightGr]{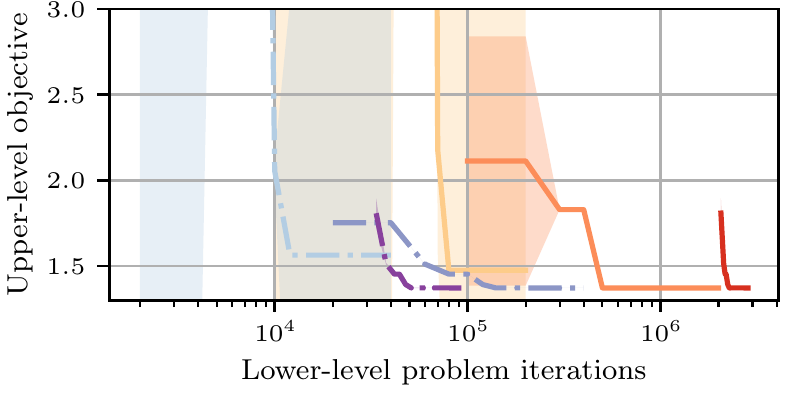}}
    \\
    \subfloat[Parameter value $\alpha_{\theta}$]{\label{fig_denoising1_comparison_param}\includegraphics[width=6cm, height=\HeightGr]{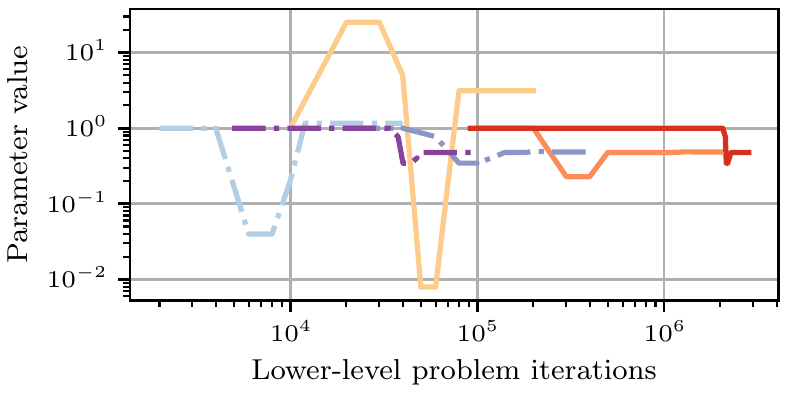}}
    \hspace{1cm}
    \subfloat[Cumulative GD/FISTA iterations per upper-level evaluation]{\label{fig_denoising1_comparison_iters_cumulative}\includegraphics[width=6cm, height=\HeightGr]{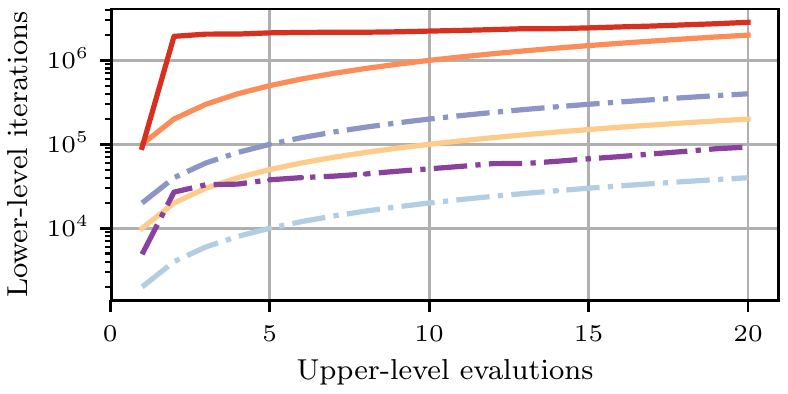}}
	\caption{Results for the 1-parameter denoising problem.}
	\label{fig_denoising1_comparison}

\vspace{\floatsep}
    \includegraphics[width=8cm,height=3.5cm]{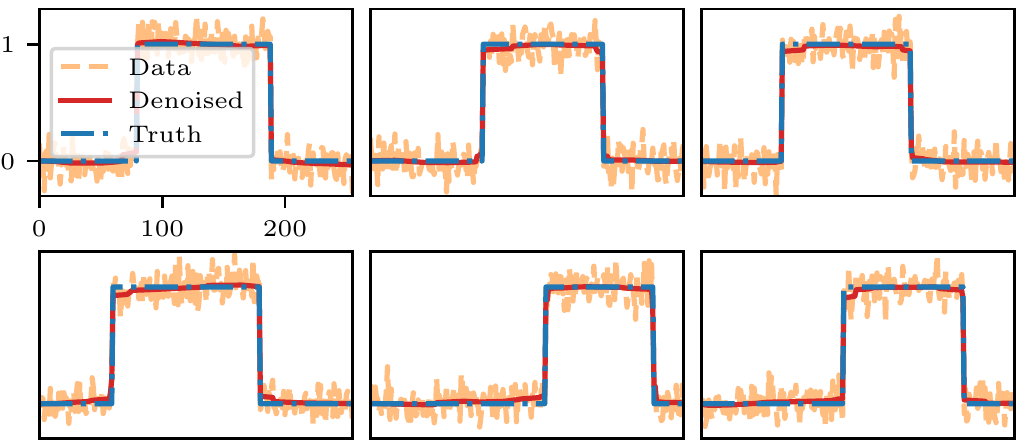}
	\caption{1-parameter final reconstructions (dynamic accuracy FISTA but all except low-accuracy GD look basically the same). Reconstructions are calculated by using the final $\theta$ returned from the given DFO-LS variant, and solving \eqref{EQ:LOWER:APPROX} with 1,000 iterations of FISTA.}
	\label{fig_denoising1_reconstructions}
\end{figure}

\begin{figure}
    \centering
    \subfloat[Start $\theta^0=1$]{\includegraphics[width=6cm, height=\HeightGr]{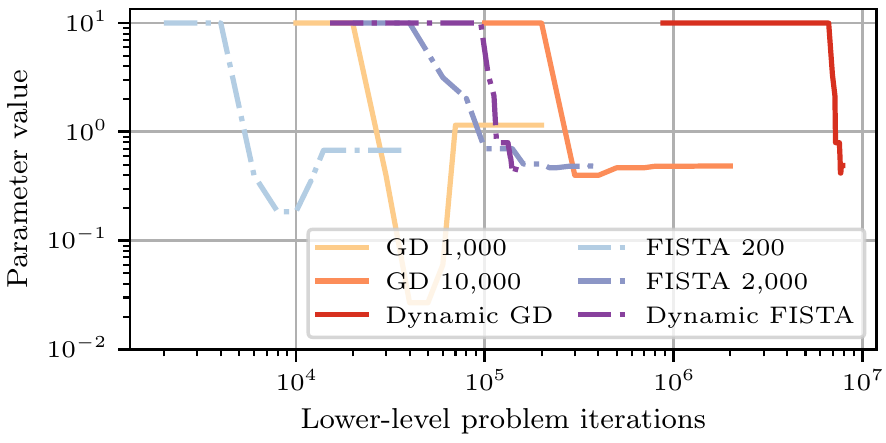}}
    \hfill
    \subfloat[Start $\theta^0=-1$]{\includegraphics[width=6cm, height=\HeightGr]{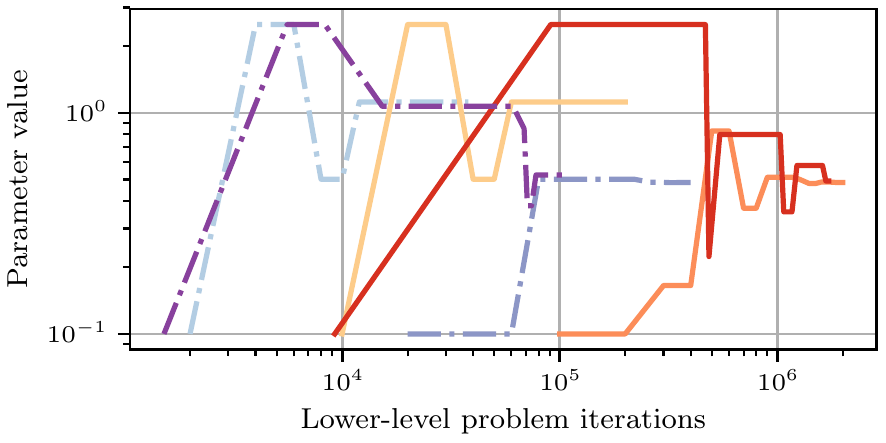}}
    \hfill
    \subfloat[Start $\theta^0=-2$]{\includegraphics[width=6cm, height=\HeightGr]{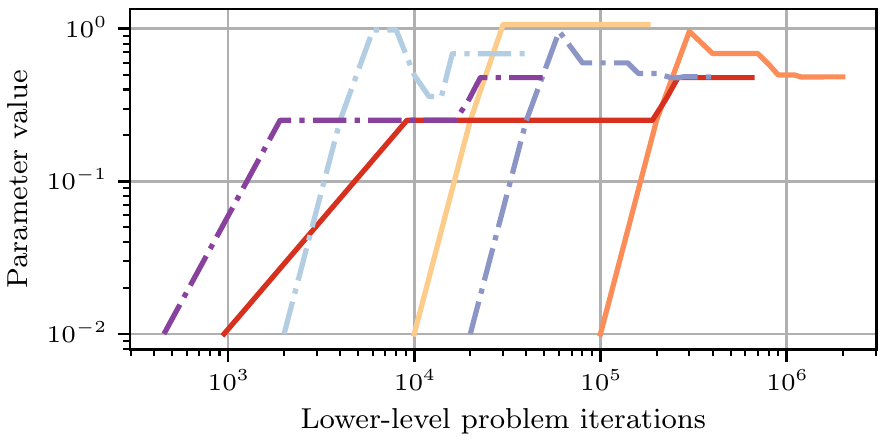}}
    \caption{1-parameter results: optimal $\alpha_{\theta}$ found when using different initial values $\theta^0$ (compare \figref{fig_denoising1_comparison_param}).}
	\label{fig_denoising1_starting_point}
	
\vspace{\floatsep}
    \includegraphics[width=8cm, height=2.5cm]{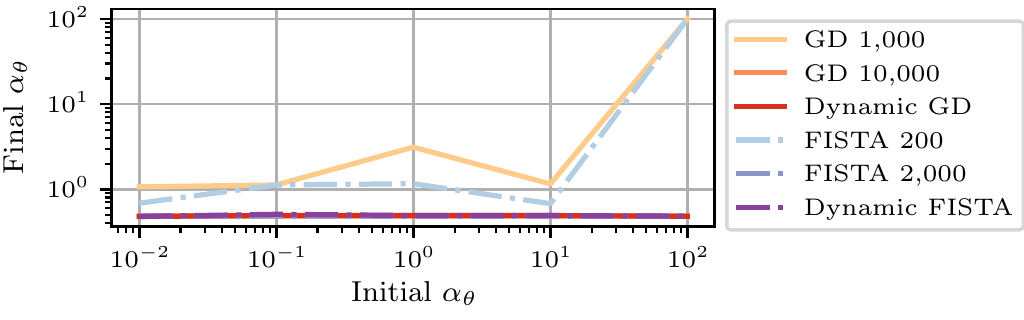}
    \caption{1-parameter results: compare optimal $\alpha_{\theta}$ values found for different choices of starting points.}
	\label{fig_denoising1_starting_point_compare}
\end{figure}

We see that both low-accuracy variants do not converge to the optimal $\theta$.
Both high-accuracy variants converge to the same objective value and $\theta$, but take much more computational effort to do this.
Indeed, we did not know a priori how many GD/FISTA iterations would be required to achieve convergence.
By contrast, both dynamic accuracy variants find the optimal $\theta$ without any tuning.

Moreover, dynamic accuracy FISTA converges faster than high-accuracy FISTA, but the reverse is true for GD.
In \figref{fig_denoising1_comparison_iters_cumulative} we show the cumulative number of GD/FISTA iterations performed after each evaluation of the upper-level objective.
We see that the reason for dynamic accuracy GD converging slower than than high-accuracy GD is that the initial upper-level evaluations require many GD iterations; the same behavior is seen in dynamic accuracy FISTA, but to a lesser degree.
This behavior is entirely determined by our (arbitrary) choices of $\theta^0$ and $\Delta^0$.
We also note that the number of GD/FISTA iterations required by the dynamic accuracy variants after the initial phase is much lower than both the fixed accuracy variants.
\revision{The difference between the GD and FISTA behavior in \figref{fig_denoising1_comparison_iters_cumulative} is based on how the initial dynamic accuracy requirements compares to the chosen number of high-accuracy iterations (10,000 GD or 2,000 FISTA).}
Finally, in \figref{fig_denoising1_reconstructions} we show the reconstructions achieved using the $\alpha_{\theta}$ found by dynamic accuracy FISTA.
All reconstructions are close to the ground truth, with a small loss of contrast.

To further understand the impact of the initial evaluations and the robustness of our framework, in \figref{fig_denoising1_comparison} we run the same problem with different choices $\theta^0\in\{-2,-1,1\}$ (where $\theta^0=0$ before). 
In \figref{fig_denoising1_starting_point} we show best $\alpha_{\theta}$ found for a given computational effort for these choices.
When $\theta^0>0$, the lower-level problem is starts more ill-conditioned, and so the first upper-level evaluations for the dynamic accuracy variants require more GD/FISTA iterations.
However, when $\theta^0<0$, we initially have a well-conditioned lower-level problem, and so the dynamic accuracy variants require many fewer GD/FISTA iterations initially, and they converge at the same or a faster rate than the high-accuracy variants.

These results also demonstrate that the dynamic accuracy variants give a final regularization parameter which is robust to the choice of $\theta^0$.
In \figref{fig_denoising1_starting_point_compare} we plot the final learned $\alpha_{\theta}$ value compared to the initial choice of $\alpha_{\theta}$ for all variants.
The low-accuracy variants do not reach a consistent minimizer for different starting values, but the dynamic and high-accuracy variants both reach the same minimizer for all starting points.
Thus although our upper-level problem is nonconvex, we see that our dynamic accuracy approach can produce solutions which are robust to the choice of starting point.

\paragraph{3-parameter denoising results}
Next, we consider the 3-parameter ($\alpha_{\theta}$, $\nu_{\theta}$ and $\xi_{\theta}$) denoising problem.
\begin{figure}
\centering
    \subfloat[Objective value $f(\theta)$]{\label{fig_denoising3_obj}\includegraphics[width=6cm, height=\HeightGr]{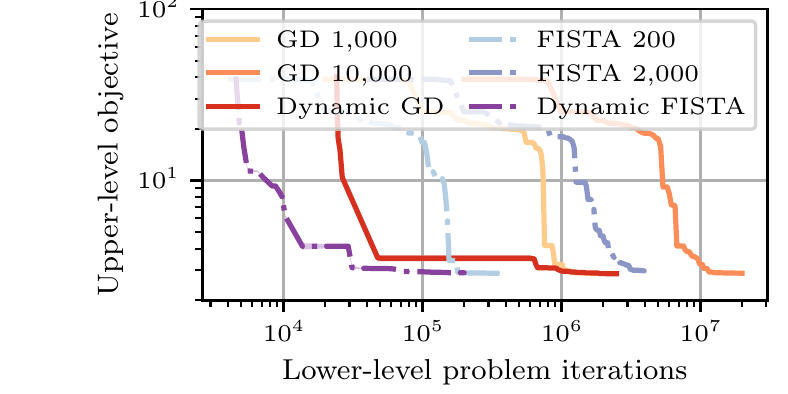}}
    \\
    \subfloat[Cumulative GD/FISTA iterations per upper-level evaluation]{\label{fig_denoising3_iters_cumulative}\includegraphics[width=6cm, height=\HeightGr]{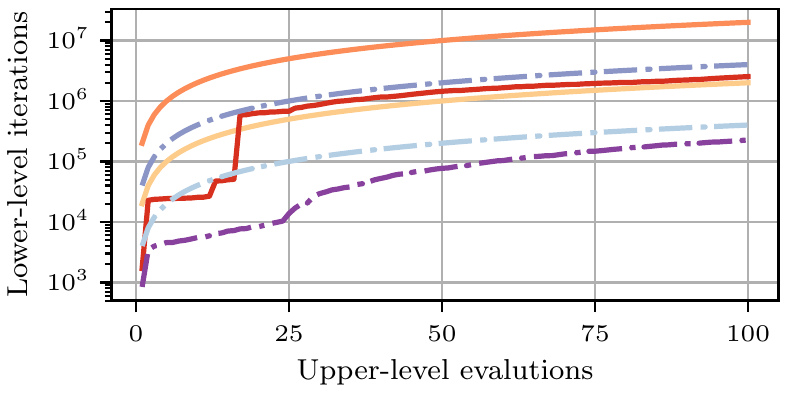}}
	\caption{Results for the 3-parameter 1D denoising problem.}
	\label{fig_denoising3}
\vspace{\floatsep}
\vspace{-0.3cm}
    \includegraphics[width=8cm, height=3.5cm]{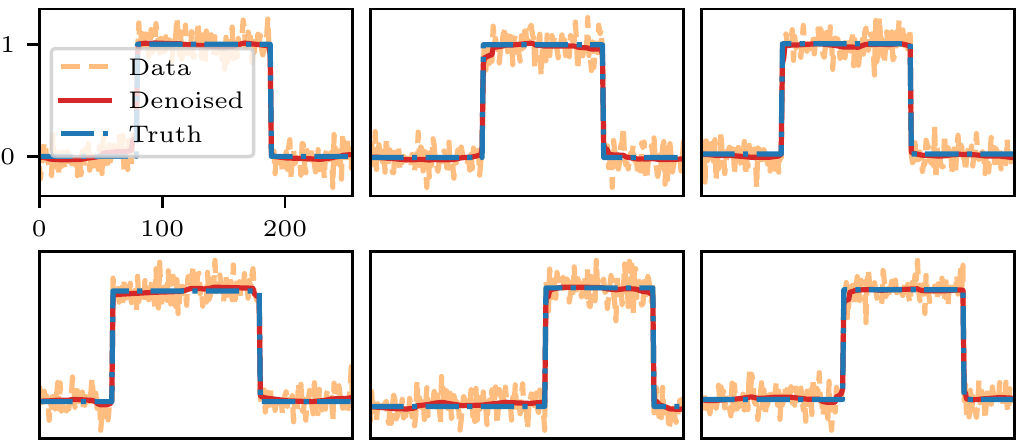}
	\caption{Example 3-parameter final reconstructions (dynamic accuracy FISTA but all other variants are similar). Reconstructions use the final $\theta$ returned by DFO-LS and solving \eqref{EQ:LOWER:APPROX} with 1,000 FISTA iterations.}
	\label{fig_denoising3_recons}
\vspace{\floatsep}
\vspace{-0.3cm}
    \includegraphics[width=8cm, height=3.5cm]{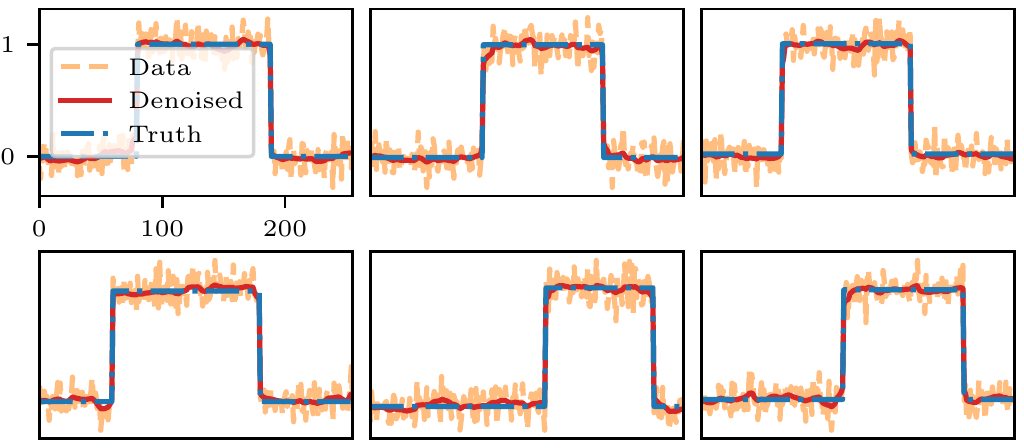}
    \caption{Example 3-parameter final reconstructions for dynamic accuracy FISTA with $\beta=10^{-4}$. Compare with reconstructions with $\beta=10^{-6}$ shown in \figref{fig_denoising3_recons}.}
	\label{fig_denoising_recons_alt_beta}
\vspace{\floatsep}
\vspace{-0.3cm}
    \includegraphics[width=8cm, height=3.5cm]{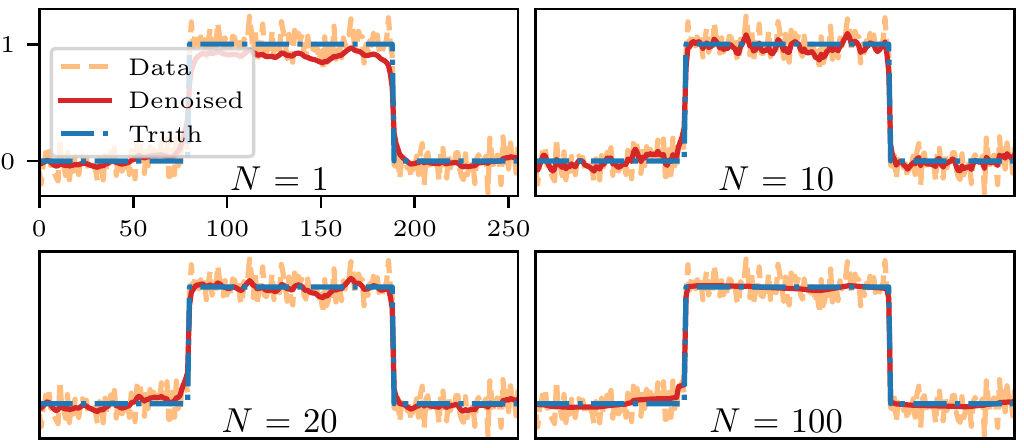}
    \caption{3-parameter results (dynamic accuracy FISTA): reconstructions of first training image using best parameters $\theta$ after $N$ evaluations of upper-level objective (reconstruction based on 1,000 FISTA iterations).}
	\label{fig_denoising_recons_by_iter}
\end{figure}
As shown in \figref{fig_denoising3}, both dynamic accuracy variants (GD and FISTA) achieve the best objective value at least one order of magnitude faster than the corresponding low- and high-accuracy variants.
We note that (for instance) 200 FISTA iterations was insufficient to achieve convergence in the 1-parameter case, but converges here.
By contrast, aside from the substantial speedup in the 3-parameter case, our approach converges in both cases without needing to select the computational effort in advance.

The final reconstructions achieved by the optimal parameters for dynamic accuracy FISTA are shown in \figref{fig_denoising3_recons}.
We note that all variants produced very similar reconstructions (since they converged to similar parameter values), and that all training images are recovered with high accuracy.

Next, we consider the effect of the upper-level regularization parameter $\beta$.
If the smaller $\beta$ value of $10^{-8}$ is chosen, all variants converge to slightly smaller values of $\nu_{\theta}$ and $\xi_{\theta}$ as the original $\beta=10^{-6}$, but produce reconstructions of a similar quality.
However, increasing the value of $\beta$ yields parameters which give noticeably worse reconstructions.
The reconstructions for $\beta=10^{-4}$ are shown in \figref{fig_denoising_recons_alt_beta}.

We conclude by demonstrating in \figref{fig_denoising_recons_by_iter} that, aside from reducing our upper-level objective
, the parameters found by DFO-LS do in fact progressively improve the quality of the reconstructions. 
The figure shows the reconstructions of one training image achieved by the best parameters found (by the dynamic accuracy FISTA variant) after a given number of upper-level objective evaluations.
We see a clear improvement in the quality of the reconstruction as the upper-level optimization progresses.

\subsection{Application: 2D denoising}
Next, we demonstrate the performance of dynamic accuracy DFO-LS on the same 3-parameter denoising problem from \secref{sec_denoising}, but applied to 2D images.
Our training data are the 25 images from the Kodak dataset.\footnote{Available from \url{http://www.cs.albany.edu/~xypan/research/snr/Kodak.html}.}
We select the central $256\times 256$-pixel region of each image, convert to monochrome and add Gaussian noise $N(0,\sigma^2)$ with $\sigma=0.1$ to each pixel independently.
We run DFO-LS for 200 upper-level evaluations with $\rho_{\rm end}=10^{-6}$.
Unlike \secref{sec_denoising}, we find that there is no need to regularize the upper-level problem with the condition number of the lower-level problem (i.e.~$\mathcal{J}(\theta)=0$ for these results).

The resulting objective decrease, final parameter values and cumulative lower-level iterations are shown in \figref{fig_2d_denoising3}.
All variants achieve the same (upper-level) objective value and parameter $\alpha_{\theta}$, but the dynamic accuracy variants achieve this with substantially fewer GD/FISTA iterations compared to the low- and high-accuracy variants.
Interestingly, despite all variants achieving the same upper-level objective value, they do not reach a consistent choice for $\nu_{\theta}$ and $\xi_{\theta}$.

In \figref{fig_2d_recons} we show the reconstructions achieved by the dynamic accuracy FISTA variant for three of the training images.
We see high-quality reconstructions in each case, where the piecewise-constant reconstructions favored by TV regularization are evident.

Lastly, we study the impact of changing the noise level $\sigma$ on the calibrated total variational regularization parameter $\alpha_{\theta}$.
We run DFO-LS with dynamic accuracy FISTA for 200 upper-level evaluations on the same training data, but corrupted with noise level $\sigma$ ranging from $10^{-1}$ (as above) to $10^{-8}$, see \figref{fig_sigma_alpha}.
We see that as $\sigma\to 0$, so does $\alpha_{\theta}$ and $\sigma^2/\alpha_\theta$. Note that this is a common assumption on the parameter choice rule in regularization theory to yield a \textit{convergent} regularization method \cite{Scherzer2008book,Ito2014book}. It is remarkable that the learned optimal parameter also has this property.

\begin{figure}
    \centering
    \subfloat[Objective value $f(\theta)$]{\label{fig_2d_denoising3_obj}\includegraphics[width=6cm, height=\HeightGr]{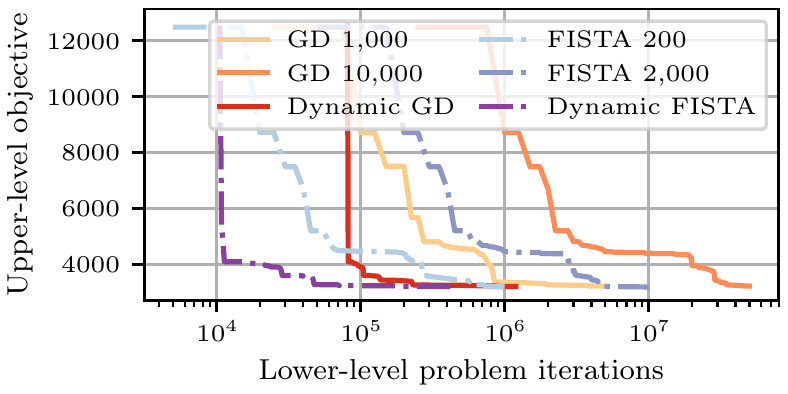}}
    \hspace{1cm}
    \subfloat[Cumulative GD/FISTA iterations per upper-level evaluation]{\label{fig_2d_denoising3_iters_cumulative}\includegraphics[width=6cm, height=\HeightGr]{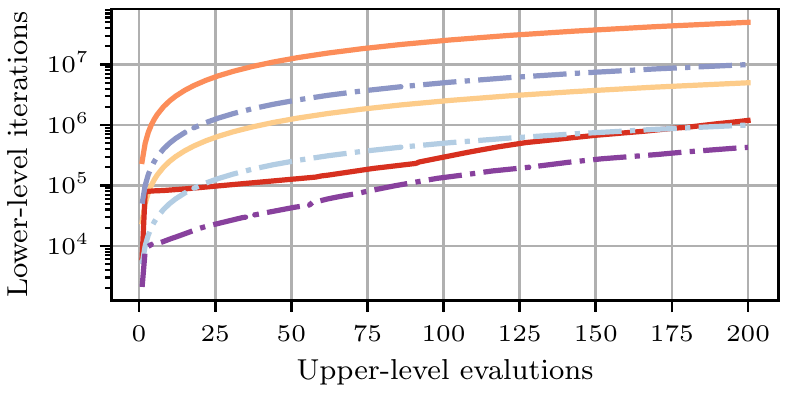}}\\
    \subfloat[Parameter value $\alpha_{\theta}$]{\label{fig_2d_denoising3_param0}\includegraphics[width=6cm, height=\HeightGr]{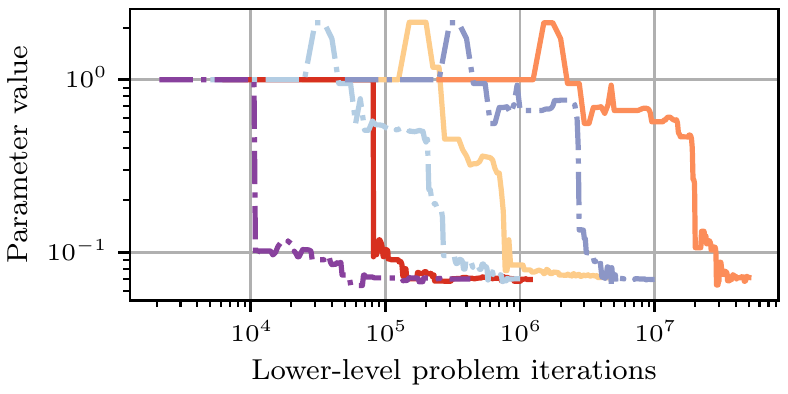}}
    \hfill
    \subfloat[Parameter value $\nu_{\theta}$]{\label{fig_2d_denoising3_param1}\includegraphics[width=6cm, height=\HeightGr]{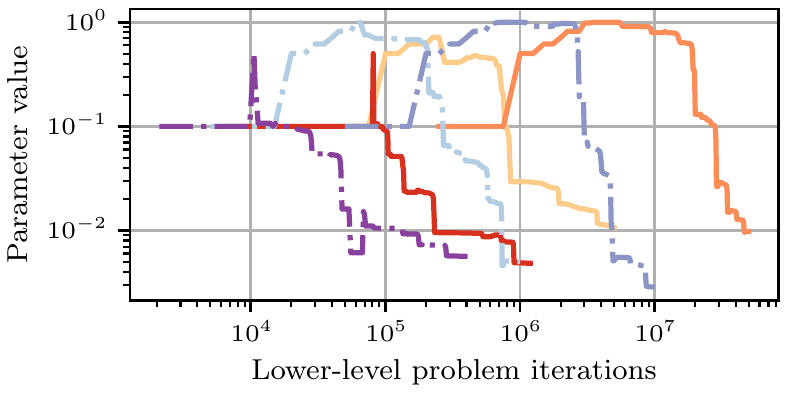}}
    \hfill
    \subfloat[Parameter value $\xi_{\theta}$]{\label{fig_2d_denoising3_param2}\includegraphics[width=6cm, height=\HeightGr]{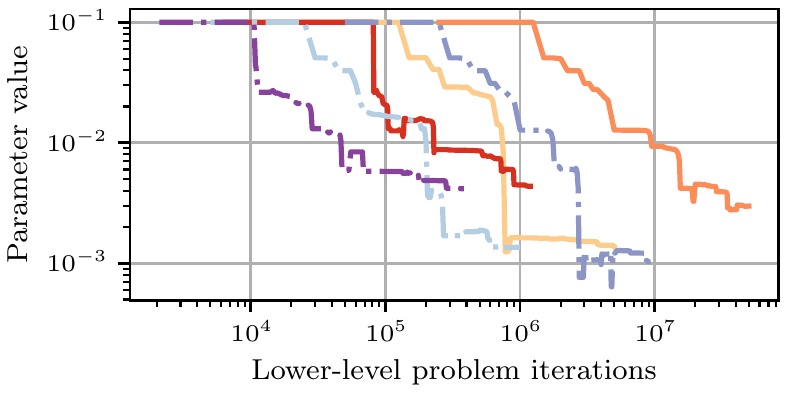}}
	\caption{Results for the 3-parameter denoising problem with 2D images.}
	\label{fig_2d_denoising3}
\end{figure}

\begin{figure}
    \centering
    \includegraphics[height=15cm,width=6cm]{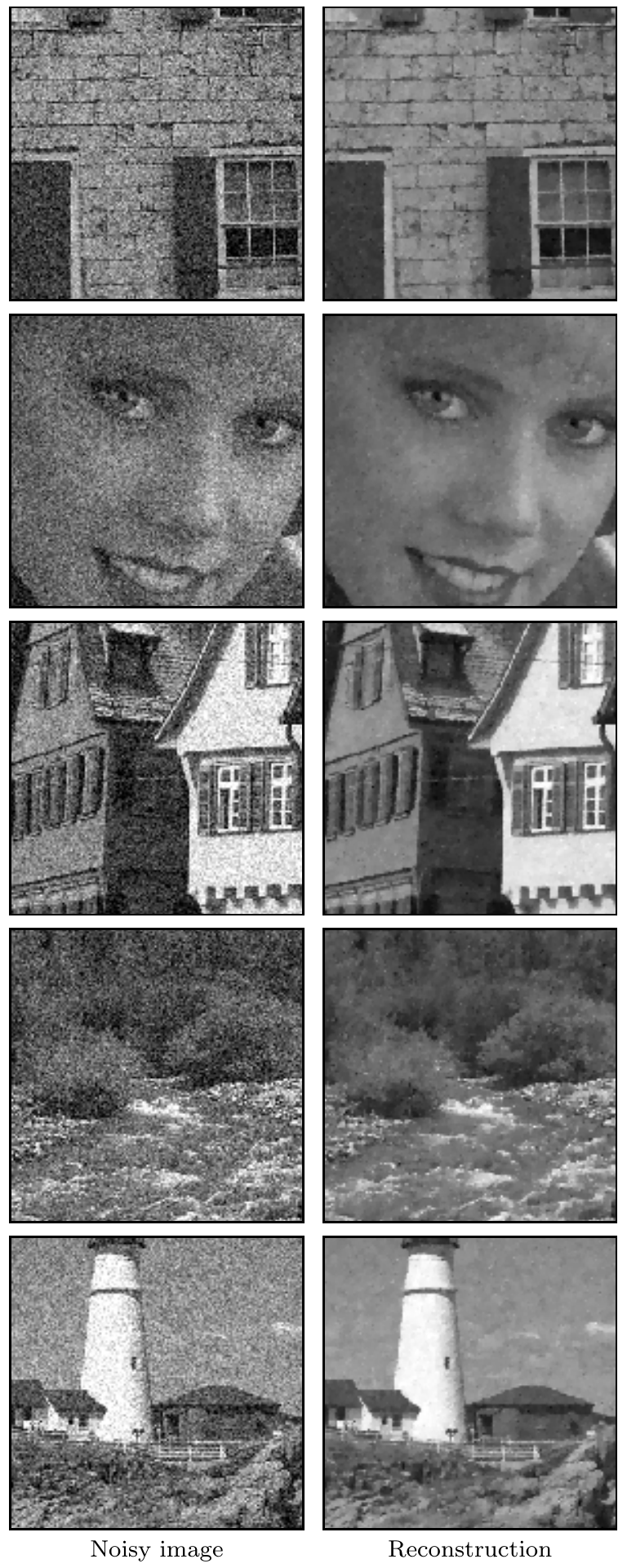}
	\caption{Example reconstructions using denoising parameters (dynamic FISTA DFO-LS variant). Reconstructions generated with 2,000 FISTA iterations of the lower-level problem.}
	\label{fig_2d_recons}
\vspace{\floatsep}
    \includegraphics[width=7cm,height=2.5cm]{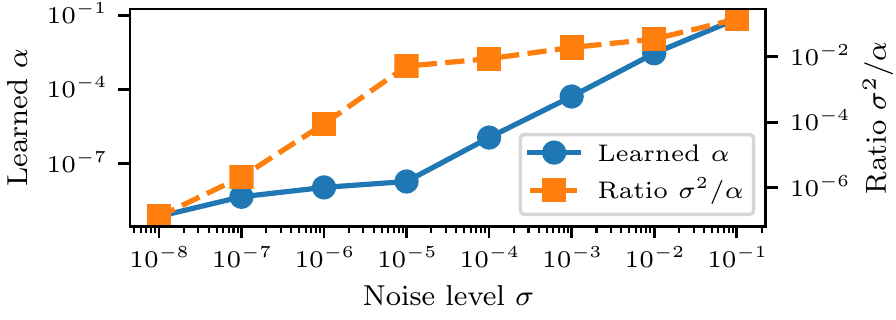}
	\caption{Learned regularization parameter $\alpha_{\theta}$ for 2D TV-denoising with varying noise levels $\sigma$.}
	\label{fig_sigma_alpha}
\end{figure}

\subsection{Application: Learning MRI Sampling Patterns} \label{sec_learning_mri_results}
\begin{figure}
    \centering  
    \subfloat[Objective value $f(\theta)$]{\label{fig_mri_obj}\includegraphics[width=6cm, height=\HeightGr]{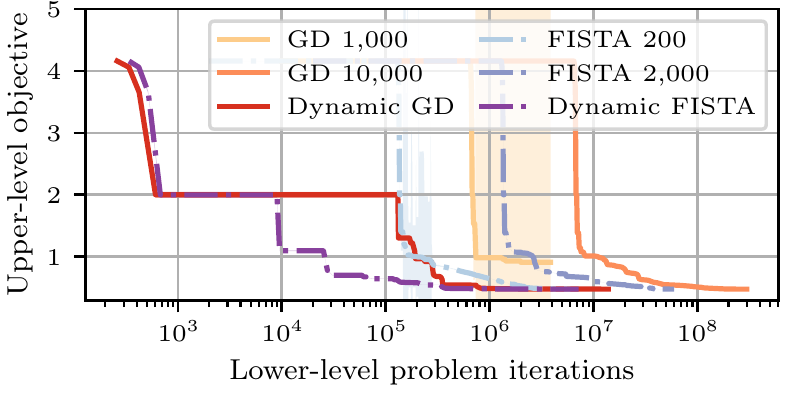}}
    \\
    \subfloat[Cumulative GD/FISTA iterations per upper-level evaluation]{\label{fig_mri_iters_cumulative}\includegraphics[width=6cm, height=\HeightGr]{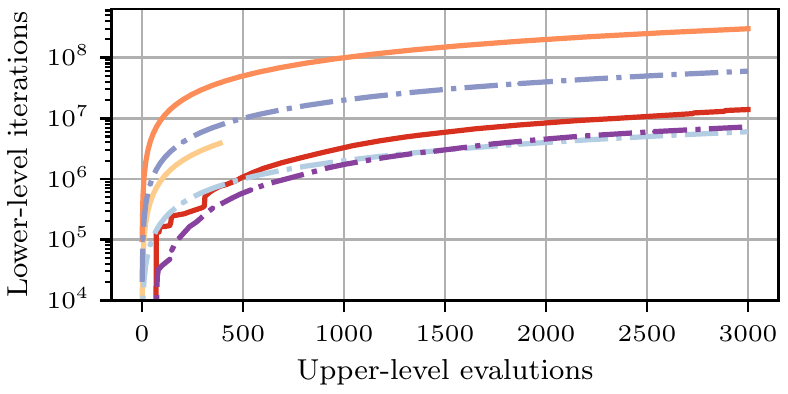}}
	\caption{Results for the MRI sampling problem. Note: the low-accuracy GD variant ($K=1,000$) terminates on a small trust-region radius, as it is unable to make further progress.}
	\label{fig_mri_results}
\vspace{\floatsep}
     \includegraphics[width=8cm, height=4cm]{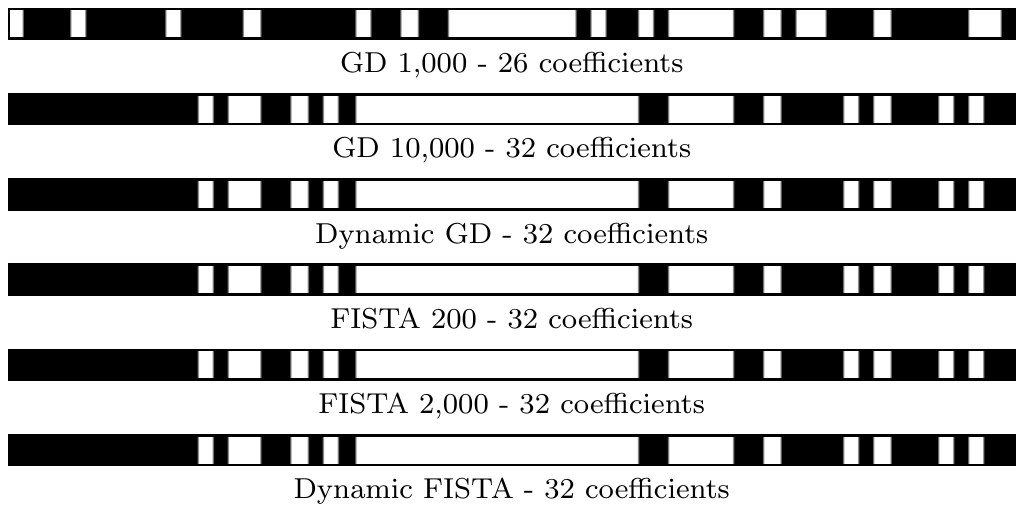}
	\caption{Final (after thresholding) MRI sampling patterns found by each DFO-LS variant. This only shows which Fourier coefficients have $\theta_i>0.001$, it does not show the relative magnitudes of each $\theta_i$.}
	\label{fig_mri_sampling}
\vspace{\floatsep}
    \includegraphics[width=8cm, height=3.5cm]{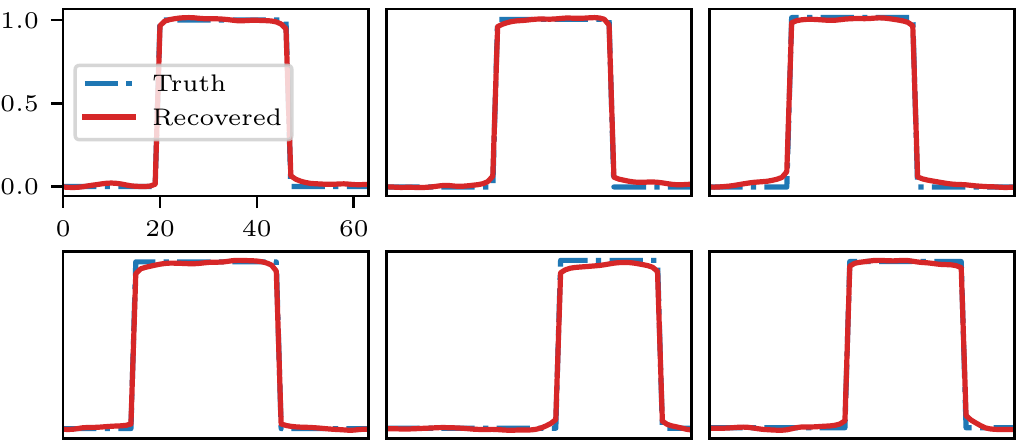}
	\caption{Reconstructions using final (after thresholding) MRI sampling pattern found by the dynamic FISTA variant of DFO-LS. Results from running 2,000 FISTA iterations of the lower-level problem.}
	\label{fig_mri_recons}
\end{figure}
Lastly, we turn our attention to the problem of learning MRI sampling patterns.
In this case, our lower-level problem is \eqref{EQ:LOWER:APPROX} with $A(\theta)=F$, where $F$ is the Fourier transform, and $S(\theta)$ is a nonnegative diagonal sampling matrix.
Following \cite{Chen2014}, we aim to find sampling parameters $\theta\in[0,1]^d$ corresponding to the weight associated to each Fourier mode, our sampling matrix is defined as
\begin{align}
    S(\theta) := \operatorname{diag}\left(\frac{\theta_1}{1-\theta_1}, \ldots, \frac{\theta_d}{1-\theta_d}\right) \in\R^{d\times d}.
\end{align}
The resulting lower-level problem is $\mu$-strongly convex and $L$-smooth as per \eqref{eq_condition_number} with $\|A_\theta^* S_\theta A_\theta\| = \|S(\theta)\| = \max_i \theta_i/(1-\theta_i)$ and $\lambda_{\operatorname{min}}(A_\theta^* S_\theta A_\theta) = \min_i \theta_i/(1-\theta_i)$. 

For our testing, we fix the regularization and smoothness parameters $\alpha=0.01$, $\nu=0.01$ and $\xi=10^{-4}$ in \eqref{EQ:LOWER:APPROX}.
We use $n=10$ training images constructed using the method described in \secref{sec_data_model} with $N=64$ and $\sigma=0.05$.
Lastly, we add a penalty to our upper-level objective to encourage sparse sampling patterns: $\mathcal{J}(\theta):=\beta \|\theta\|_1$, where we take $\beta=0.1$.
To fit the least-squares structure \eqref{eq_upper_exact}, we rewrite this term as $\mathcal{J}(\theta)=(\sqrt{\beta\|\theta\|_1})^2$.
To ensure that $S(\theta)$ remains finite and $\mathcal{J}(\theta)$ remains $L$-smooth, we restrict $0.001 \leq \theta_i \leq 0.99$.

We run DFO-LS with a budget of 3000 evaluations of the upper-level objective and $\rho_{\rm end}=10^{-6}$.
As in \secref{sec_denoising}, we compare dynamic accuracy DFO-LS against (fixed accuracy) DFO-LS with low- and high-accuracy evaluations given by a 1,000 and 10,000 iterations of GD or 200 and 1,000 iterations of FISTA.

With our $\ell_1$ penalty on $\theta$, we expect DFO-LS to find a solution where many entries of $\theta$ are at their lower bound $\theta_i=0.001$.
Our final sampling pattern is chosen by using the corresponding $\theta_i$ if $\theta_i>0.001$, otherwise we set that Fourier mode weight to zero.

In \figref{fig_mri_results} we show the objective decrease achieved by each variant, and the cumulative lower-level work required by each variant.
All variants except low-accuracy GD achieve the best objective value with low uncertainty.
However, as above, the dynamic accuracy variants achieve this value significantly earlier than the fixed accuracy variants, largely as a result of needing much fewer GD/FISTA iterations in the (lower accuracy) early upper-level evaluations.
In particular dynamic accuracy GD reaches the minimum objective value about 100 times faster than high-accuracy GD.
We note that FISTA with 200 iterations ends up requiring fewer lower-level iterations after a large number of upper-level evaluations, but the dynamic accuracy variant achieves is minimum objective value sooner.

We show the final pattern of sampled Fourier coefficients (after thresholding) in \figref{fig_mri_sampling}.
Of the five variants which found the best objective value, all reached a similar set of `active' coefficients $\theta_i>0.001$ with broadly similar values for $\theta_i$ at all frequencies.
For demonstration purposes we plot the reconstructions corresponding to the coefficients from the `dynamic FISTA' variant in \figref{fig_mri_recons} (the reconstructions of the other variants were all similar).
All the training images are reconstructed to high accuracy, with only a small loss of contrast near the jumps.

\section{Conclusion}
We introduce a dynamic accuracy model-based DFO algorithm for solving bilevel learning problems.
This approach allows us to learn potentially large numbers of parameters, and allowing inexact upper-level objective evaluations with which we dramatically reduce the lower-level computational effort required, particularly in the early phases of the algorithm.
\revision{Compared to fixed accuracy DFO methods, we often achieve better upper-level objective values and low-accuracy methods, and similar objective values as high-accuracy methods but with much less work: in some cases up to 100 times faster.}
These observations can be made for both lower-level solvers GD and FISTA, with different fixed accuracy requirements, for ROF-denoising and learning MRI sampling patterns.
Thus the proposed approach is robust in practice, computationally efficient and backed by convergence and worst-case complexity guarantees.
Although the upper-level problem is nonconvex, our numerics do not suggest that convergence to non-global minima is a point for concern here.

Future work in this area includes relaxing the smoothness and/or strong convexity assumptions on the lower-level problem (making the upper-level problem less theoretically tractable).
Our theoretical analysis would benefit from a full proof that our worst-case complexity bound on the lower-level computational work is tight.
Another approach for tackling bilevel learning problems would be to consider gradient-based methods which allow inexact gradient information.
Lastly, bilevel learning appears to compute a regularization parameter choice strategy which yields a convergent regularization method. Further investigation is required to back these numerical results by sound mathematical theory.

\printbibliography
\end{document}